\documentclass[11pt,leqno]{amsart}
\usepackage{epsfig}
\usepackage{amssymb}
\usepackage{amscd}
\usepackage[matrix,arrow]{xy}
\usepackage{graphicx}

\newtheorem{theo}{Theorem}[section]
\newtheorem{lemma}[theo]{Lemma}
\newtheorem{defi}[theo]{Definition}
\newtheorem{prop}[theo]{Proposition}
\newtheorem{conj}[theo]{Conjecture}
\newtheorem{cor}[theo]{Corollary}
\newtheorem{remark}[theo]{Remark}

\numberwithin{equation}{section}

\def\A{{\mathbb A}}

\def\bL{\mathbb{L}}
\def\R{\mathbb{R}}
\def\C{\mathbb{C}}
\def\Z{\mathbb{Z}}
\def\Q{\mathbb{Q}}

\def\bR{{\mathbf R}}
\def\bL{{\mathbf L}}

\def\PP{{\mathbb P}}

\def\pre-tr{\operatorname{pre-tr}}

\def\Hom{\operatorname{Hom}}
\def\End{\operatorname{End}}
\def\gr{\operatorname{gr}}

\newcommand{\cF}{{\mathcal F}}

\newcommand{\cO}{{\mathcal O}}

\newcommand{\cM}{{\mathcal M}}

\newcommand{\cD}{{\mathcal D}}

\newcommand{\cA}{{\mathcal A}}

\newcommand{\cC}{{\mathcal C}}

\newcommand{\cT}{{\mathcal T}}

\newcommand{\cH}{{\mathcal H}}

\newcommand{\un}{\underline}

\newcommand{\mrB}{\operatorname{B}}

\newcommand{\Add}{\operatorname{Add}}
\newcommand{\Perf}{\operatorname{Perf}}
\newcommand{\Perv}{\operatorname{Perv}}
\newcommand{\rat}{\operatorname{rat}}

\newcommand{\coker}{\operatorname{Coker}}

\newcommand{\im}{\operatorname{Im}}
\newcommand{\re}{\operatorname{Re}}

\newcommand{\Rep}{\operatorname{Rep}}
\newcommand{\Crit}{\operatorname{Crit}}
\newcommand{\Nilp}{\operatorname{Nilp}}
\newcommand{\Mat}{\operatorname{Mat}}

\newcommand{\Ext}{\operatorname{Ext}}

\newcommand{\Tr}{\operatorname{Tr}}

\newcommand{\Arg}{\operatorname{Arg}}

\newcommand{\Aut}{\operatorname{Aut}}
\newcommand{\rmMod}{\operatorname{Mod}}

\newcommand{\Spec}{\operatorname{Spec}\,}

\newcommand{\Ho}{\operatorname{Ho}}

\newcommand{\id}{\operatorname{id}}

\newcommand{\Gr}{\operatorname{Gr}}

\newcommand{\pr}{\operatorname{pr}}

\usepackage{epsf}
\usepackage{amscd}
\newcommand{\rmmod}{\operatorname{mod}}

\title[Quantum cluster variables via vanishing cycles]
{Quantum cluster variables via vanishing cycles}

\author{Alexander I. Efimov}
\address{Steklov Mathematical Institute of RAS, Gubkin str. 8, GSP-1, Moscow 119991, Russia}
\email{efimov@mccme.ru}
\thanks{The author was partially supported by
"Dynasty" Foundation, Simons Foundation, RFBR (grant 4713.2010.1), and by AG Laboratory HSE, RF government grant, ag. 11.G34.31.0023.}

\begin{document}

\begin{abstract}In this paper, we provide a Hodge-theoretic interpretation of Laurent phenomenon for general skew-symmetric quantum cluster algebras, using Donaldson-Thomas theory for a quiver with potential. It turns out that the positivity conjecture reduces to
the certain statement on purity of monodromic mixed Hodge structures on the cohomology with the coefficients in the sheaf of vanishing cycles on the moduli of stable framed representations.

As an application, we show that the positivity conjecture (and actually a stronger result on Lefschetz property) holds if either initial or mutated quantum seed is acyclic. For acyclic initial seed the positivity has been already shown by F. Qin \cite{Q} in the quantum case, and also by Nakajima \cite{Nak} in the commutative case.\end{abstract}


\maketitle

\tableofcontents

\section{Introduction}

Cluster algebras were introduced in \cite{FZ02}. They form a certain class of commutative algebras with a distinguished
set of generators, which are called {\it cluster variables.} If the cluster algebra has rank $n,$ then the set of generators
is a union of distinguished $n$-element subsets called {\it clusters.} There is a rule of mutation of such clusters,
when one cluster variable is replaced by some very simple rational function in the variables of the same cluster:
$$xx'=M_1+M_2,$$
where $x$ is the cluster variable, $x'$ is its replacement, and $M_1$ and $M_2$ are monomials in the other cluster variables
in the same cluster. Moreover, all clusters are obtained by such mutations from any given cluster.

The most surprising property of cluster algebras is {\it Laurent phenomenon:} any cluster variable
is actually a Laurent polynomial in the variables of any given cluster. It leads to the well-known {\it positivity conjecture:}
all such Laurent polynomials have non-negative integer coefficients.

We suggest \cite{Ke1}, \cite{Ke2} as nice survey articles on cluster algebras and their categorification.

In \cite{P1}, Plamondon obtains uses certain categorification of cluster algebras to obtain a general formula for cluster monomials for skew-symmetric cluster algebras.
The same formulas are actually obtained in \cite{DWZ2}, the coincidence is shown in \cite{P2}.
The resulting coefficients are Euler characteristics of some quiver Grassmannians. However, this does not imply the positivity conjecture,
since a priori Euler characteristic can be negative.

Another approach to categorification is studied by Nakajima \cite{Nak}. He uses it to prove positivity conjecture (w.r.t. all seeds) for cluster algebras coming from bipartite quivers.
In the Appendix of \cite{Nak} the positivity conjecture is proved for acyclic initial seed. It was announced by Y. Kimura and F. Qin \cite{KQ} that they
have a generalization of Nakajima's results on categorification for all acyclic {\it quantum} cluster algebras (see below the definition of quantum cluster algebras), which implies positivity conjecture w.r.t.
all quantum seeds which are mutationally equivalent to an acyclic seed.

Our approach is very close to the paper of K. Nagao \cite{N}. He uses Donaldson-Thomas theory (in a framework different from our paper) for quivers with potentials to obtain the same formulas for cluster variables (which can be easily generalized for
cluster monomials), under certain assumptions on the quiver with potential: the potential should be polynomial and this property should be preserved
under a finite sequence of mutations. The reason for such restrictive assumption is that the Donaldson-Thomas theory is not well-developed at the moment
for the case of formal potential (see below).

The goal of this paper is to obtain the formulas for quantum cluster monomials, using the approach of \cite{N}, for
 arbitrary skew-symmetric quantum cluster algebras. Also, our results can be viewed as a generalization of quantum cluster character
\cite{Q}. We use the framework of mixed Hodge modules for Donaldson-Thomas theory,
  which is developed in \cite{KS}, and has first been considered in \cite{DS} in the geometric situation.
Positivity conjecture does not follow automatically from our results (as it does in \cite{Q} for acyclic initial seed),
but it reduces to a certain conjecture on purity of monodromic mixed Hodge structures (see below).

We will deal with skew-symmetric quantum cluster algebras, as introduced in \cite{BZ}. Every skew-symmetrizable cluster algebra
can be quantized in the sense of \cite{BZ}. Here the algebra is non-commutative: it is contained in the skew-field of fractions
of a quantum torus. If $L\cong\Z^m$ is some free finitely generated abelian group, and
$$\Lambda:L\times L\to\Z$$
a non-degenerate pairing, then the quantum torus $\cT_{\Lambda}$ is an algebra over $\Z[q^{\pm\frac12}],$
with a distinguished basis $X^e,$ $e\in L,$ satisfying
$$X^e\cdot X^f=q^{\frac12\Lambda(e,f)}X^{e+f}.$$
The algebra $\cT_{\Lambda}$ is an Ore domain, and we have its skew-field of fractions $\cF_{\Lambda}.$
A cluster is assigned to the quantum seed $(M,\tilde{B}),$ where $M:\Z^m\to\cF_{\Lambda}$
is a map of special kind, the so-called toric frame (the analogue of transcendence basis, generating $\cF_{\Lambda}$), and $\tilde{B}\in\Mat_{m\times n}(Z)$
(where $n$ is the rank of quantum cluster algebra), satisfying some compatibility condition (see Subsections \ref{ss:TF}, \ref{ss:SSQS}).
 There is a rule of mutation of such quantum seeds (see Subsection \ref{ss:MSSQS}).

 The cluster variables are $M(e_i),$ $1\leq i\leq n\leq m,$
the cluster monomials are $M(\lambda),$ $\lambda\in\Z_{\geq 0}^m,$ and the elements $M(e_i),$ $n+1\leq m,$ are called {\it coefficients,}
 they do not change under mutations.
Taking the $\Z[q^{\pm\frac12}]$subalgebra $\cA_S\subset\cF_{\Lambda}$ generated by all cluster variables, coefficients and their inverses
in a given mutation-equivalence class $S$ of quantum seeds, we obtain the {\it quantum cluster algebra} (Definition \ref{def:QCA}).

Again, the Laurent phenomenon holds: if $L=\Z^m,$ and there is some quantum seed $(M,\tilde{B})$ in $S,$ with $M(c)=X^{c},$
then $$\cA_S\subset\cT_{\Lambda}$$
(Theorem \ref{LP}). We also have positivity conjecture: all cluster monomials have positive coefficients as elements of $\cT_{\Lambda}.$

We will obtain a Hodge-theoretic interpretation of quantum Laurent phenomenon for all cluster monomials,
using Donaldson-Thomas theory for a quiver with potential, developed in \cite{KS}. First, we replace the base ring
$\Z[q^{\frac12}]$ by $\hat{R}[T^{\frac12}],$ where $\hat{R}$ is the completion of $R=K_0(MMHS),$ where $MMHS$
is the abelian category of so-called monodromic mixed Hodge structures with Thom-Sebastiani product $\star_+,$ and $T$ is the class of $\Q(-1)$ (see Subsection \ref{ss:MMHS}).

For any smooth algebraic variety $X,$ with a regular
function $f:X\to\C$ and a locally closed subset $X^{sp}\subset f^{-1}(0),$ one defines (Definition \ref{critical_cohomology}) the critical cohomology with compact support
$$H^{i,crit}_c(X^{sp},f)\in MMHS.$$

We have the motivic quantum torus $$\cT_{\Lambda}^{mot}:=\cT_{\Lambda}\otimes_{\Z[q^{\pm\frac12}]}\hat{R}[T^{\frac12}],$$
and motivic cluster monomials $M'^{mot}(\lambda)=M'(\lambda)\otimes 1\in\cT_{\Lambda}^{mot}.$

For any quiver $Q,$ polynomial potential $W$ on $Q,$ a central charge $Z:\Z^{V(Q)}\to\C$ (given by some map $V(Q)\to\cH_+$) an angle
$0<\phi<\pi,$ and a vector $\lambda\in\Z_{\geq 0}^{V(Q)}$ one associates the moduli space of stable framed representations
$\cM_{\gamma,<\phi,\lambda}^{sfr},$ which is a smooth algebraic variety with a regular function
$$W_{\gamma}=\Tr(W):\cM_{\gamma,<\phi,\lambda}^{sfr}\to\C$$
(Subsection \ref{ss:SFR}). We have a closed subset
$$\cM_{\gamma,<\phi,\lambda}^{sp,sfr}\subset \cM_{\gamma,<\phi,\lambda}^{sfr}$$
which consists of stable framed representations which are nilpotent and are critical points of $W_{\gamma}.$

Now suppose that $L=\Z^m,$ the initial seed is of the form $(M(c)=X^c,\tilde{B}),$ and we have the mutated seed
$$(M_r,\tilde{B}_r)=\mu_{k_r}(\dots(\mu_{k_1}(M,\tilde{B}))).$$
One can associate to the matrix $\tilde{B}$ a (non-unique) quiver $Q$ with the set of vertices $\{1,\dots,m\},$ and define the mutations for quivers (Subsection \ref{ss:MSSQS}).

Let $Q$ be the quiver corresponding to $\tilde{B},$ and $$Q_r=\mu_{k_r}(\dots(\mu_{k_1}(Q))\dots).$$

The notion of mutation can be extended to quivers with formal potentials (QP's) \cite{DWZ}, see Subsection \ref{ss:MQFP}. Suppose that the polynomial QP $(Q_r,W_r)$
is obtained by a sequence of mutations from a formal QP $(Q,W):$
$$(Q_r,W_r)=\mu_{k_r}(\dots(\mu_{k_1}((Q,W)))\dots).$$
We have Ginzburg DG algebras \cite{G} $\hat{\Gamma}_{Q,W},$ $\hat{\Gamma}_{Q_r,W_r},$ inclusions
$$\iota:D^b(\hat{\Gamma}_{Q,W})\hookrightarrow \Perf(\hat{\Gamma}_{Q,W})\subset D(\hat{\Gamma}_{Q,W}),$$
natural isomorphisms
$$K_0(\Perf(\hat{\Gamma}_{Q,W}))\cong\Z^m,\quad K_0(D^b(\hat{\Gamma}_{Q,W}))\cong\Z^m,$$
and similarly for $\hat{\Gamma}_{Q_r,W_r}$ (Subsection \ref{ss:GDGA}). We have a natural choice of equivalence
$$\Phi(r):D(\hat{\Gamma}_{Q,W})\stackrel{\sim}{\to} D(\hat{\Gamma}_{Q_r,W_r})$$
coming from Nagao's result \cite{N} (Theorem \ref{Nagao} below). Also, put
$$\hat{\Gamma}_{\un{k},\lambda}:=\Phi(r)^{-1}(\bigoplus\limits_{j=1}^m\hat{\Gamma}_{Q_r,W_r,j}^{\lambda_j}),\quad\lambda\in\Z_{\geq 0}^m,$$
where $\hat{\Gamma}_{Q_r,W_r,j}$ are indecomposable projective DG modules.

As usual, for any finite quiver $Q$ with the set of vertices $V(Q),$ we denote by $\chi_Q$ the Euler form on $\Z^{V(Q)}:$
$$\chi_Q(\gamma_1,\gamma_2)=\sum\limits_{i\in V(Q)}\gamma_1^i\gamma_2^j-\sum\limits_{i,j\in V(Q)}a_{ji}\gamma_1^i\gamma_2^j,$$
where $a_{ij}$ is the number of arrows from $i$ to $j.$

\begin{theo}\label{LP_by_sfr_intro} There exist a central charge $Z$ on $Q_r,$
 and an angle $0<\phi<\pi,$ such that for $\lambda\in\Z_{\geq 0}^m$
$$M_r^{mot}(\lambda)=X^{[\hat{\Gamma}_{\un{k},\lambda}]}\cdot\sum\limits_{\gamma\in\Z_{\geq 0}^m}[D(H^{\bullet,crit}_c(\cM_{\gamma,<\phi,\lambda}^{sp,sfr},(W_r)_{\gamma}))]\cdot T^{-\frac12\chi_{Q_r}(\gamma,\gamma)}X^{\iota(\Phi(r)^{-1}[1](\gamma))},$$
where we view $\gamma$ as an element of $K_0(D^b(\hat{\Gamma}_{Q_r,W_r})).$
In particular, we have that
$$[H^{\bullet,crit}_{c}(\cM_{\gamma,<\phi,\lambda}^{sp,sfr},(W_r)_{\gamma}))]\in\Z[T^{\pm 1}]\subset R.$$\end{theo}

More precise formulation is Theorem \ref{LP_by_sfr}.

In the terminology of \cite{FZ07}, the vectors $[\hat{\Gamma}_{\un{k},\lambda}]$ are $g$-vectors, and the polynomials
$$\sum\limits_{\gamma\in\Z_{\geq 0}^m}[D(H^{\bullet,crit}_c(\cM_{\gamma,<\phi,\lambda}^{sp,sfr},(W_r)_{\gamma}))]\cdot T^{-\frac12\chi_{Q_r}(\gamma,\gamma)}y^{\Phi(r)^{-1}[1](\gamma)}$$
are motivic quantum $F$-polynomials.

\begin{cor}In the assumptions of the above Theorem, suppose that for some $\lambda$ and for all $\gamma$ we have that
$H^{i,crit}_c(\cM_{\gamma,\pi-\phi_r,\lambda}^{sp,sfr},(W_r)_{\gamma})\in MMHS$
is pure of weight $i.$ Then the cluster monomial $M_r(\lambda)$ is a positive element in $\cT_{\Lambda}.$\end{cor}

For any quiver $Q$ with the set of vertices $\{1,\dots,m\},$ denote by $Q_{[1,\dots,n]}$ its full subquiver on the vertices $1,\dots,n.$

\begin{defi}The quantum seed $(M,\tilde{B})$ is called acyclic, if the quiver $Q_{[1,\dots,n]}$ is acyclic.\end{defi}

As a consequence, we obtain the following result.

\begin{theo}\label{acyclic_intro}In the above notation, suppose that either $(M,\tilde{B})$ or $(M_r,\tilde{B}_r)$ is acyclic quantum seed.
Then for any $\lambda\in \Z_{\geq 0}^m,$ we have that $H^{i,crit}_c(\cM_{\gamma,<\phi,\lambda}^{sp,sfr},(W_r)_{\gamma})$
is pure of weight $i.$ In particular, by the above Corollary, we have that
all cluster monomials $M_r(\lambda)$ are positive elements of $\cT_{\Lambda}.$\end{theo}

Actually, below we formulate a Conjecture \ref{conj_Lefschetz}, which is stronger than positivity conjecture,
and prove it under the assumptions of Theorem \ref{acyclic_intro} (Theorem \ref{positivity_for_acyclic}).

The paper is organized as follows.

In Section \ref{s:QCA} we recall the definition of skew-symmetric quantum cluster algebras (Definition \ref{def:QCA}), formulate Laurent Phenomenon (Theorem \ref{LP})
and Positivity Conjecture (Conjecture \ref{positivity_conj}). Here we also propose a stronger conjecture (for skew-symmetric case) about Lefschetz property
(Conjecture \ref{conj_Lefschetz}).

Section \ref{s:DTQPP} is devoted mostly to an overview of DT theory for a quiver with polunomial potential, in the framework of mixed Hodge modules.
 Following \cite{KS}, we recall moduli of quiver representations (Subsection \ref{ss:GN}), the category of monodromic mixed Hodge structures
 (Subsection \ref{ss:MMHS}), Donaldson-Thomas series (Subsection \ref{ss:DTS}) and factorization theorem for them (Theorem \ref{factorization_polynomial}). Here we also  recall stable framed representations (Subsection \ref{ss:SFR}), and obtain a formula relating them to DT series (Theorem \ref{conj_and_sfr}). Result of this kind is actually standard in Donaldson-Thomas theory.

In Section \ref{s:CQCA} we recall various notions and results on categorification of (quantum) cluster algebras: Ginzburg DG algebras (Subsection \ref{ss:GDGA}), mutations of quivers with potentials and decorated representations (Subsections \ref{ss:MQFP}, \ref{ss:DRM}),
derived equivalences of Keller and Yang between Ginzburg DG algebras of two QP's related by a mutation (Subsection \ref{ss:DE}). We recall the tilting of t-structures using torsion pairs (Subsection \ref{ss:Tilting}),  also the theorem of Nagao (Theorem \ref{Nagao}) about the natural derived equivalence arising from the sequence of mutations.
In Subsection \ref{ss:Plamondon} we recall the result of Plamondon (Theorem \ref{Plamondon}), and explain how it is related to the theorem of Nagao (Lemma \ref{Nagao_Plamondon}).

We also recall $3$CY  $A_{\infty}$-categories associated with quivers with potentials (Subsection \ref{ss:3CY}) and explain
how they can be used to extend DT theory to some quivers with formal potentials, in particular, to QP'S which are mutation equivalent to polynomial QP's
(Subsection \ref{ss:DTQFP}). We end this section with explaining that torsion pairs arising naturally from a sequence of mutations (from Theorem \ref{Nagao}) actually come from some central charges and decompositions of the upper half-plane into disjoint union of two sectors (Theorem \ref{torsion_pairs_via_Z}).

Section \ref{s:LPSFR} is devoted to the proof of Theorem \ref{LP_by_sfr_intro}, the more precise formulation is Theorem \ref{LP_by_sfr}.
The central charge and the angle $\phi$ come from Theorem \ref{torsion_pairs_via_Z}. First, we obtain a formula for cluster monomials using conjugation
by certain DT series (Theorem \ref{cluster_var_via_conj}) and then, using Theorem \ref{conj_and_sfr}, we obtain the desired expression in terms of stable framed representations.

In Section \ref{s:PCVP} we explain that the positivity conjecture reduces to a conjecture on purity of critical cohomology (the first half of Conjecture \ref{purity_conj}), and the stronger conjecture on Lefschetz property (Conjecture \ref{conj_Lefschetz}) reduces to the (conjectural) existence of Lefschetz operator on critical cohomology (the second half of Conjecture \ref{purity_conj}). We easily show that purity and existence of Lefschetz operator holds in the case when either initial or mutated seed is acyclic (Theorem \ref{positivity_for_acyclic}), in particular proving Theorem \ref{acyclic_intro}. We also give an example when the relevant mixed Hodge module of vanishing cycles is not pure, but the critical cohomology is however pure.

In Section \ref{s:conj_on_exc} we conjecture that for the relevant moduli of stable framed representations, there exists an exceptional collection
in (the Karoubian completion of) the homotopy category of matrix factorizations. This conjecture implies the first half of conjecture \ref{purity_conj}, namely the purity.

In Appendix we prove Theorem \ref{DT_for_formal}, which allows to use DT theory for some quivers with formal potentials.

{\bf Acknowledgements.} I am grateful to Maxim Kontsevich and Yan Soibelman for useful discussions.

\section{Quantum cluster algebras}
\label{s:QCA}

In this section we introduce skew-symmetric quantum cluster algebras, formulate the Laurent Phenomenon and the Positivity conjecture, following
\cite{BZ}.

\subsection{Quantum torus}
\label{ss:QT}

Let $L\cong\Z^m$ be a free abelian group of rank $m,$ and $\Lambda:L\times L\to\Z$ be a skew-symmetric form. We associate to
the pair $(L,\Lambda)$ the based quantum torus $\cT_{\Lambda}.$ It is a $\Z[q^{\pm \frac12}]$-algebra, with the $\Z[q^{\pm \frac12}]$-basis
$X^e,$ $e\in L.$ The product is given by the formula
$$X^e\cdot X^f=q^{\frac{\Lambda(e,f)}2}X^{e+f}.$$
Clearly, we have commutation relations
$$X^eX^f=q^{\Lambda(e,f)}X^fX^e.$$

Associativity is obvious. Further, we have that the ring $\cT_{\Lambda}$ satisfies (left and right) Ore condition, hence we have its skew-field of fractions $\cF_{\Lambda},$
which contains central subfield $\Q(q^{\frac12}).$

\subsection{Toric frames}
\label{ss:TF}

A toric frame is a map $M:\Z^m\to \cF_{\Lambda},$ which is of the form
$$M(c)=\varphi(X^{\eta(c)}),$$
where $\eta:\Z^m\to L$ is an isomorphism of lattices, and $\varphi\in\Aut_{\Q(q^{\frac12})}(\cF_{\Lambda}).$

Consider the skew-symmetric integral form $\Lambda_M=\Lambda_{\eta}$ on $\Z^m,$ transferred from $L$ by $\eta.$
It is clear that we have $$M(c)M(d)=q^{\frac{\Lambda_{\eta}(c,d)}2}M(c+d),\quad M(c)M(d)=q^{\Lambda_{\eta}(c,d)}M(d)M(c).$$
In particular, $M$ is uniquely defined by $X_i=M(e_i),$ where $e_i\in\Z^m$ are standard basis elements.

On the other hand, given $\eta$ (and hence $\Lambda_{\eta}$) we can formulate necessary and sufficient conditions on elements
$X_1,\dots,X_m\in\cF_{\Lambda}$ to give rise to a toric frame:

$(i)$ elements $X_1,\dots,X_m$ are invertible and we have $$X_iX_j=q^{\Lambda_{\eta}(e_i,e_j)}X_jX_i;$$

$(ii)$ the induced map $\cT_{\Lambda_{\eta}}\to \cF_{\Lambda},$
$$X^c\mapsto (q^{-\frac12})^{\sum\limits_{1\leq i<j\leq m}\Lambda_{\eta}(e_i,e_j)c_ic_j}X_1^{c_1}\dots X_m^{c_m},$$ is injective;

$(iii)$ elements $X_1,\dots,X_m$ generate the skew-field $\cF_{\Lambda}$ over $\Q(q^{\frac12}).$

Then $$M(c)=(q^{-\frac12})^{\sum\limits_{1\leq i<j\leq m}\Lambda_{\eta}(e_i,e_j)c_ic_j}X_1^{c_1}\dots X_m^{c_m}$$
is a toric frame.

\subsection{Skew-symmetric quantum seeds}
\label{ss:SSQS}

Fix an integer $1\leq n\leq m.$ A skew-symmetric quantum seed is a pair $(M,\tilde{B}),$
where $M$ is a toric frame for $\cF_{\Lambda},$ and $\tilde{B}\in \Mat_{m\times n}(\Z)$ is a matrix such that
$$\tilde{B}^t\Lambda_M=\widetilde{I_n},$$
where $\widetilde{I_n}\in\Mat_{n\times m}(\Z)$ is a matrix with identity block $n\times n$ and zero block $n\times (m-n).$
Note that the upper $n\times n$ submatrix $B$ of $\tilde{B}$ is necessarily skew-symmetric, since $$B=\widetilde{I_n}\tilde{B}=\tilde{B}^t\Lambda_M\tilde{B}.$$

For a quantum seed $(M,\tilde{B}),$ we have the set ${\bf\tilde{X}}=\{X_1,\dots,X_m\},$ where $X_i=M(e_i).$ The subset ${\bf X}=\{X_1,\dots,X_n\}$
is called a {\it cluster} associated to $(M,\tilde{B}),$ and the elements of ${\bf X}$ are called {\it cluster variables}.
The elements $M(\lambda),$ $\lambda\in\Z_{\geq 0}^m,$ are called {\it cluster monomials}. The elements of ${\bf C}={\bf \tilde{X}}\setminus {\bf X}$ are called {\it coefficients}.

\subsection{Mutations of skew-symmetric quantum seeds}
\label{ss:MSSQS}

Let $(M,\tilde{B})$ be a skew-symmetric quantum seed as above, and choose an integer $1\leq k\leq n.$ We recall that mutation $\mu_k(M,\tilde{B})=(M',\tilde{B}')$
is defined as follows. First,
$$b'_{ij}:=\begin{cases}-b_{ij} & \text{if }i=k\text{ or }j=k;\\
b_{ij}+\frac{|b_{ik}|b_{kj}+b_{ik}|b_{kj}|}2.\end{cases}$$
Second, in order to define $M'$ it suffices to define $M'(e_i),$ $1\leq i\leq m.$ Put
$$M'(e_i)=\begin{cases}M(e_i) & \text{for }i\ne k;\\
M(\sum\limits_{b_{ik}>0}b_{ik}e_i-e_k)+M(-\sum\limits_{b_{ik}<0}b_{ik}e_i-e_k) & \text{for }i=k.\end{cases}$$

Note that the set of coefficients ${\bf C}$ is invariant under mutations.

\begin{defi}\label{def:QCA}Let $S$ be some equivalence class of skew-symmetric quantum seeds under mutations.
Then the cluster algebra $\cA_S$ is the $\Z[q^{\pm\frac12}]$-subalgebra of $\cF_{\Lambda}$ generated by the union of all clusters associated to quantum seeds in $S,$
the coefficients and their inverses.\end{defi}

We will use the following replacement of matrices $\tilde{B}.$ Namely, first, skew-symmetric integral matrices $B\in\Mat_{n\times n}(\Z)$ are in bijection with quivers $Q$ such that the vertex set $V(Q)$ is identified with $\{1,\dots,n\},$ and there are no loops or oriented $2$-cycles in $Q.$ The mutation of matrices
corresponds to natural mutations of quivers. Further, for a matrix $\tilde{B}\in\Mat_{m\times n}(\Z)$ such that the upper submatrix $B\in\Mat_{n\times n}(\Z)$
is skew-symmetric, we can take a (non-unique) quiver $\tilde{Q}$ with $V(\tilde{Q})=\{1,\dots,m\},$ without loops and $2$-cycles, such that the corresponding
matrix of the numbers of arrows $(a_{ij})_{1\leq i,j\leq m}$ is related to $\tilde{B}$ by the formula
$$a_{ji}-a_{ij}=b_{ij},\quad 1\leq i\leq m,\,1\leq j\leq n.$$

The mutations of $\tilde{B}$ correspond to mutations of $\tilde{Q}$ at the vertices $1,\dots,n.$ We denote by $\tilde{Q}_{[1,\dots,n]}$
the full subquiver of $\tilde{Q}$ on the vertices $1,\dots,n.$

\begin{defi}The quantum seed $(M,\tilde{B})$ is called acyclic if the quiver $\tilde{Q}_{[1,\dots,n]}$ is acyclic.\end{defi}

\subsection{Laurent phenomenon and positivity conjecture}
\label{ss:LPPC}

The following result was proved by Berenstein and Zelevinsky \cite{BZ} for a more general class of skew-symmetrizable quantum cluster algebras.

\begin{theo}\label{LP}Assume that $L=\Z^m,$ and suppose that equivalence class $S$ of skew-symmetric quantum seeds contains a representative (initial quantum seed)
$(M,\tilde{B}),$ where $M(c)=X^c,$ $c\in\Z^m.$ Then the quantum cluster algebra $\cA_S$ is contained in $\cT_{\Lambda}.$\end{theo}

\begin{conj}\label{positivity_conj}In the assumptions of the above theorem, all cluster monomials are positive, i.e. of the form $\sum\limits_{c}P_c(q^{\frac12})X^c,$ where $P_c(q^{\frac12})\in\Z[q^{\pm \frac12}]$ are Laurent polynomials with non-negative coefficients.\end{conj}

Of course, positivity conjecture reduces to cluster variables, but it is more natural to consider all cluster monomials.

We would like to propose a stronger conjecture. Namely, we say that the polynomial $P(q^{\frac12})\in \Z[q^{\pm 12}]$
satisfies {\it Lefschetz property}, if there is some $N\in\Z$ such that $P$ is the non-negative linear combination of polynomials
$$P=\sum\limits_{k\in\Z}c_k P(N,k)=q^{\frac{N}2}(q^{\frac{-k}2}+q^{\frac{2-k}2}+\dots+q^{\frac{k-2}2}+q^{\frac{k}2}),$$
where all the numbers $k\in\Z_{\geq 0},$ for which $c_k$ is non-zero, have the same parity.

\begin{conj}\label{conj_Lefschetz}In the notation of the Conjecture \ref{positivity_conj}, the polynomials $P_c(q^{\frac12})$ satisfy Lefschetz property.\end{conj}

Below we will obtain the following result (see Theorem \ref{positivity_for_acyclic}) as an application of the general formula (Theorem \ref{LP_by_sfr}) for quantum cluster monomials.

\begin{theo}Suppose that either initial quantum seed $(M(c)=X^c,\tilde{B})$ or the mutated (by a sequence of mutations) quantum seed
$(M',\tilde{B}')$ is acyclic. Then the statement of Conjecture \ref{conj_Lefschetz}
holds for all cluster monomials $M'(\lambda),$ $\lambda\in\Z_{\geq 0}^m.$\end{theo}

\section{Donaldson-Thomas theory for a quiver with polynomial potential}
\label{s:DTQPP}

In this section we recall the Donaldson-Thomas theory for a polynomial QP (quiver with potential), mostly following Kontsevich and Soibelman \cite{KS}.

\subsection{General notions}
\label{ss:GN}

Let $Q$ be a finite quiver with a set of vertices $V(Q),$ and $a_{ij}\in\Z_{\geq 0}$ arrows from $i$ to $j.$
Denote by $E(Q)$ the set of arrows (edges).

We will always consider representations as {\it right} modules over the path algebra $\C Q$.
Take some dimension vector $\gamma=(\gamma^i)_{i\in V(Q)}\in\Z_{\geq 0}^I.$
We have an affine space of complex representations in coordinate spaces $\C^{\gamma^i}:$ $$M_{\gamma}=\prod\limits_{i,j\in V(Q)}\Hom(\C^{\gamma^j},\C^{\gamma^i})^{a_{ij}}=\C^{\sum\limits_{i,j\in V(Q)}a_{ij}\gamma^j\gamma^i}.$$
Further, we have an algebraic group
$G_{\gamma}=\prod\limits_{i\in I}GL(\gamma^i,\C)$ acting on $M_{\gamma}$ by conjugation.

Now, take some polynomial potential on $Q,$ i.e. an element
 $$W\in HH_0(\C Q)=\C Q/[\C Q,\C Q].$$ In other words, $W$ is finite linear combination of cyclic paths $$W=\sum\limits_{\sigma}c_{\sigma}\sigma.$$
 We assume that $c_{\sigma}=0$ if $\sigma$ is a path of length zero. In particular, if the quiver $Q$ is acyclic, then $W=0.$

Then, $W$ defines a regular function on $M_{\gamma}:$
$$W_{\gamma}=Tr(W):M_{\gamma}\to\C.$$
Clearly, we have $W_{\gamma}\in\C[M_{\gamma}]^{G_{\gamma}}.$

Now, take the Jacobi algebra of $(Q,W):$
$$J_{Q,W}=\C[Q]/\langle \partial_aW,\,a\in E(Q)\rangle,$$
where we take the quotient by the two-sided ideal generated by partial cyclic derivatives of $W$ by all edges $a.$

We have a closed $G_{\gamma}$-invariant subset $\Rep(J_{Q,W})_{\gamma}\subset M_{\gamma}$ which consists of all representations on which
cyclic derivatives of $W$ act by zero. It is easy to check that
$$\Rep(J_{Q,W})_{\gamma}=\Crit(W_{\gamma})=\{dW_{\gamma}=0\}.$$

Also, we have a closed $G_{\gamma}$-invariant subset $\Nilp_{\gamma}\subset M_{\gamma}$ which consists of nilpotent representations. Recall that a representation $E$ of a quiver $Q$ is called nilpotent if all sufficiently long paths of $Q$ act by zero on $E.$ The intersection $$\Nilp_{\gamma}\cap \Crit(W_{\gamma})\subset \Crit(W_{\gamma})\cap W_{\gamma}^{-1}(0).$$ is precisely the set of representations of the complete Jacobi algebra $$\hat{J}_{Q,W}=\widehat{\C Q}/\widehat{\langle \partial_aW,\,a\in E(Q)\rangle}$$ (in coordinate spaces $\C^{\gamma^i}$) with dimension vector $\gamma.$ Here we complete w.r.t. the length of paths.

Let us define the degree of the polynomial potential $W$ to be the maximal length of the cycle which contributes to $W.$
Let us say that some property holds for a generic potential $W$ if for sufficiently large $N$ there is a dense open subset in the affine space of potentials
of degree $\leq N$ for which the property holds. The following observation was suggested to me by M. Kontsevich.

\begin{prop}\label{Bertini}For a generic polynomial potential $W,$ we have an equality $$\Nilp_{\gamma}\cap \Crit(W_{\gamma})=\Crit(W_{\gamma})\cap W_{\gamma}^{-1}(0).$$\end{prop}

\begin{proof}We claim that we can find a collection of cyclic paths $\sigma_1,\dots,\sigma_l$ such that the subset $\Nilp_{\gamma}\subset M_{\gamma}$ is given by equations $\Tr(\sigma_1)=\dots=\Tr(\sigma_l)=0.$ Indeed, since $M_{\gamma}$ is Noetherian, it suffices to show that if $E\in M_{\gamma},$ and $\Tr_{E}(\sigma)=0$ for any cyclic path $\sigma,$ then $E$ is nilpotent.

We show this as follows. Take any vertex $i\in V(Q),$ and let $B_i\subset \End(E_i)$ be anon-unital subalgebra, which is the image of $(e_i(\C Q)e_i)\cap (\C Q)_+,$
 where $e_i\in\C Q$ is the idempotent at the vertex $i,$ and $(\C Q)_+\subset \C Q$ is the ideal generated by arrows. Then we have $\Tr_{E_i}(b)=0$ for any $b\in B_i.$ Hence, by Engel's Theorem, $B_i$ is contained in the subalgebra of strictly
upper triangular matrices w.r.t. some basis of $E_i.$ It follows by Dirichlet's principle that any path in $Q$ of length at least $\sum\limits_{i\in I}\gamma^i$ acts by zero on $E,$ therefore $E$ is nilpotent.

Hence, there is a finite collection of cycles $\sigma_1,\dots,\sigma_l$ with the required property. Now, put $$N_0=\max(\deg(\sigma_1),\dots,\deg(\sigma_l)).$$ The statement of the Proposition follows from Bertini's Theorem, applied to the linear systems
$$\PP(V_N),\quad V_N=\langle \Tr(\sigma),\,1\leq \deg(\sigma)\leq N\rangle\subset H^0(M_{\gamma},\cO),\quad N\geq N_0.$$\end{proof}

Now take some central charge on $Q,$ i.e. a homomorphism of abelian groups $$Z:\Z^{V(Q)}\to\C,$$ given by a map $$V(Q)\to \cH_+=\{\im z>0\}\subset\C.$$
Then, for any non-zero dimension vector $\gamma\in\Z_{\geq 0}^{V(Q)}\setminus\{0\}$ we have $Z(\gamma)\in\cH_+,$ thus we have $\Arg(Z(\gamma))\in (0,\pi).$
Now, for any non-zero finite-dimensional representation $E$ of $Q$ we define its {\it slope}
$$\phi(E):=\Arg(Z(E))\in (0,\pi),\quad Z(E):=Z(\un{\dim}\, E).$$

\begin{defi}A non-zero finite-dimensional representation $E$ of $Q$ is called stable (resp. semi-stable) w.r.t. $Z$ if for any of its proper non-zero subrepresentation $E'\subset E$ we have $\phi(E')<\phi(E)$ (resp. $\phi(E')\leq\phi(E)$).\end{defi}

The following is well-known:

\begin{prop}For any finite-dimensional representation $E$ of $Q$ we have a unique increasing filtration $0=E_0\subset E_1\subset\dots E_n=E,$ such that all subquotients $E_i/E_{i-1}$ are semi-stable, and
$$\phi(E_1)>\phi(E_2/E_1)>\dots>\phi(E_n/E_{n-1}).$$
This filtration is called Harder-Narasimhan filtration.\end{prop}

\begin{defi}Take some sector $V\subset\cH_+,$ i.e. $V$ is closed under summation and multiplication by positive reals (for example $V$ may be just a ray).
Then the open (but possibly empty) $G_{\gamma}$-invariant subset $$M_{\gamma,V}\subset M_{\gamma}$$ consists of representations
$E$ such that for the Harder-Narasimhan filtration $E_{\bullet}$ on $E$ we have $Z(E_i/E_{i-1})\in V.$\end{defi}

To see that $M_{\gamma,V}\subset M_{\gamma}$ is indeed open, note that its complement consists of representations $E\in M_{\gamma},$
such that either there exists a non-zero subrepresentation $E'\subset E$ with $\Arg (Z(E'))>\Arg(V),$ or there exists a non-zero quotient representation $E''$
of $E$ such that $\Arg(Z(E''))<\Arg(V).$ Both of these conditions are clearly closed.

\subsection{Monodromic mixed Hodge structures}
\label{ss:MMHS}

We refer the reader to \cite{S}, \cite{PS} for the introduction to M. Saito's theory of mixed Hodge modules.

Recall the category $EMHS$ of exponential mixed Hodge structures, which was used by Kontsevich and Soibelman \cite{KS}.
It is defined as some full subcategory $EMHS\subset MHM_{\A^1}$ of the category of mixed Hodge modules on the affine line.
Namely,
$$EMHS=\{M\in MHM_{\A^1}\mid \bR\Gamma(\rat(M))=0\},$$
where $\rat:MHM_{\A^1}\to\Perv(\A^1,\Q)$ is the Betti realization functor.

The category $EMHS$ is closed under the Thom-Sebastiani product $\star_+$ on $D^b(MHM_{\A^1})$:
$$M\star_+ N=sum_*(M\boxtimes N),$$
where $sum:\A^1\times\A^1\to \A^1$ is the summation morphism.

Now, the inclusion functor $i:EMHS\to MHM_{\A^1}$ has left adjoint $p:MHM_{\A^1}\to EMHS,$ such that $p\circ i=\id,$ and the composition $i\circ p=:\Pi$ is given by the formula
$$\Pi(M)=M\star_+ (j_!\Q_{G_m}(0)[1]),$$
where $j:G_m\hookrightarrow\A^1$ is the embedding.

The category $EMHS$ carries the weight filtration: for $E\in EMHS,$ put
$$W_n^{EMHS} E=\Pi(W_{n+1}E),$$
where $W_{\bullet}E$ is the usual weight filtration on $MHM_{\A^1}.$ In particular, if $S$ is an admissible
variation of mixed Hodge structures on $G_m,$ then we have the object
$$E=j_! S[1]\in EMHS,\quad W_n^{EMHS}E=j_!(W_n S)[1].$$

Objects of $EMHS$ of this kind form the subcategory
$$MMHS=\{M\in EMHS\mid M\text{ is unramified over }G_m\subset\A^1\}\subset EMHS,$$
closed under Thom-Sebastiani product $\star_+.$

\begin{defi}\label{critical_cohomology}Let $X$ be a smooth algebraic variety, $f:X\to\C$ a regular function,
and $X^{sp}\subset X^0:=f^{-1}(0)$ a locally closed subset, and let $u$ be the coordinate on $G_m.$ Then the critical cohomology with compact support $H^{\bullet,crit}_c(X^{sp},f)$ is defined by the formula
\begin{multline*}H^{i,crit}_c(X^{sp},f)=\\H^{i+1}((G_m\to\A^1)_!(X^{sp}\times G_m\to G_m)_!(X^{sp}\times G_m\to X^0\times G_m)^*\phi_{\frac{f}{u}}\Q_{X\times G_m}(0))\in MMHS.\end{multline*}
\end{defi}

Here we take the vanishing cycles functor $\phi_{\frac{f}{u}}:D^b(MHM_{X\times G_m})\to D^b(MHM_{X^0\times G_m}),$
preserving the t-structure.

If the variety $X$ equipped with the action of some affine algebraic group $G\subset GL(N),$ and $f$ and $X^{sp}$ are $G$-invariant, then the equivariant critical cohomology $H^{\bullet,crit}_{c,G}(X^{sp},f)$ is defined in the standard way.


\subsection{Donaldson-Thomas series}
\label{ss:DTS}

Now, for any central charge $Z:\Z^{V(Q)}\to\C,$ the sector $V\subset\cH_+,$ and the dimension vector $\gamma\in\Z_{\geq 0}^{V(Q)},$
we have the closed $G_{\gamma}$-invariant subset $M_{\gamma,V}^{sp}:=\Nilp_{\gamma}\cap \Crit(W_{\gamma})\cap M_{\gamma,V}.$

We have the Euler form $\chi:\Z^I\times\Z^I\to\Z,$
$$\chi(\gamma_1,\gamma_2)=\sum\limits_{i\in I}\gamma_1^i\gamma_2^i-\sum\limits_{i,j\in I}a_{ji}\gamma_1^i\gamma_2^j.$$

Put $R=K_0(MMHS).$ This is a commutative ring via the Thom-Sebastiani product $\star_+.$ We have the decreasing filtration
$$F^pR=\sum\limits_{\substack{E\in MMHS,\\W_{p-1}^{EMHS}E=0}}\Z\cdot [E]\subset R.$$
Put $$\hat{R}:=\lim\limits_{\substack{\leftarrow\\p}}R/F^pR.$$

We have the object $T=j_!\Q_{G_m}(-1)[1]\in MMHS.$ We denote by the same letter its class in $\hat{R}.$

Further, let $H^{\bullet}$ be some graded object of $MMHS,$ satisfying the assumption that for each $p\in\Z,$
$W_p^{EMHS}H^{i}=0$ for all but finitely many $i.$ Then its class in $\hat{R}$ is defined by the formula
$$[H^{\bullet}]=\sum\limits_{i\in\Z}(-1)^i[H^i]\in\hat{R}.$$

The {\it motivic quantum torus} $\cT_Q^{mot}$ is the algebra over $\hat{R}[T^{\frac12}],$ with the $\hat{R}[T^{\frac12}]$-basis
$w_{\gamma},$ $\gamma\in\Z^I.$ The multiplication is given by the formula $$w_{\gamma_1}w_{\gamma_2}=T^{-\chi(\gamma_1,\gamma_2)}w_{\gamma_1+\gamma_2}.$$

We will use the modified monomials $$\hat{w}_{\gamma}:=T^{-\frac{\chi(\gamma,\gamma)}2}w_{\gamma}\in \cT_Q^{mot},\quad \gamma\in\Z^{V(Q)}.$$
They satisfy the relations:
\begin{equation}\label{motivic_mult}\hat{w}_{\gamma_1}\cdot\hat{w}_{\gamma_2}=T^{\frac{-\chi(\gamma_1,\gamma_2)+\chi(\gamma_2,\gamma_1)}2}\hat{w}_{\gamma_1+\gamma_2},
\quad \hat{w}_{\gamma_1}\cdot\hat{w}_{\gamma_2}=T^{-\chi(\gamma_1,\gamma_2)+\chi(\gamma_2,\gamma_1)}
\hat{w}_{\gamma_2}\cdot\hat{w}_{\gamma_1}.\end{equation}

For any strict convex cone $C\subset \R^I$ (not necessarily closed) one defines the completion $\hat{\cT}_{Q,C}^{mot}$
of $\cT_Q$ by the formula
$$\hat{\cT}_{Q,C}=\sum\limits_{\gamma\in\Z^I}\prod\limits_{\gamma'\in\Z^I\cap C}\hat{R}[T^{\frac12}]\cdot \hat{w}_{\gamma+\gamma'}\subset\prod\limits_{\gamma\in\Z^I}\hat{R}[T^{\frac12}]\cdot e_{\gamma}.$$
The product on $\hat{\cT}_{Q,C}$ comes from the formula \eqref{motivic_mult}.

Now, the sector $V\subset\cH_+$ defines the cone $$C_V=\{\gamma\in\R_{>0}^I\mid Z(\gamma)\in V\}.$$
One defines the critical DT series $A_V$ by the formula
\begin{equation}\label{formula_for_A_V}A_V=1+\sum\limits_{\gamma\in C_V}[D(H_{c,G_{\gamma}}^{\bullet,crit}(M_{\gamma,V}^{sp},W_{\gamma}))]\cdot
T^{-\frac{\chi(\gamma,\gamma)}2}\hat{w}_{\gamma}\in\hat{\cT}_{Q,C_V}^{mot},\end{equation}
where $D$ is the duality functor on $MMHS.$

Kontsevich and Soibelman \cite{KS} proved the following result:

\begin{theo}\label{factorization_polynomial}Suppose that the sector $V\subset\cH_+$ is the disjoint union of two sectors $V_1,V_2\subset\cH_+,$
in the clockwise order. Then we have
$$A_V=A_{V_1}A_{V_2}.$$\end{theo}

\subsection{Stable framed representations}
\label{ss:SFR}

Fix some $0<\phi<\pi,$ and define the sector
$$V_{\leq \phi}:=\{z\in\cH_+\mid \Arg(z)\leq \phi\}\subset\cH_+,$$
and analogously $V_{<\phi},V_{\geq \phi},V_{>\phi}\subset\cH_+.$

Take any vector $\lambda\in\Z_{\geq 0}^{V(Q)}\setminus\{0\}.$ Take the projective $\C Q$-module
$P_{\lambda}=\bigoplus\limits_{i\in I}P_i^{\oplus\lambda^i},$ where $P_i$ is indecomposable projective module in the $i$-th vertex. Define smooth variety
$$M_{\gamma,\leq \phi,\lambda}^{sfr}=\{(E\in M_{\gamma,V_{\leq \phi}},\,u:P_{\lambda}\to E)\mid \coker(u)\in M_{\gamma',V_{>\phi}}\text{ for some }\gamma'\},$$
and analogously $M_{\gamma,<\phi,\lambda}^{sfr},$ both equipped with $G_{\gamma}$-action. Clearly, we have $G_{\gamma}$-equivariant closed subsets
$$M_{\gamma,\leq \phi,\lambda}^{sp,sfr}\subset M_{\gamma,\leq \phi,\lambda}^{sfr},\quad M_{\gamma,< \phi,\lambda}^{sp,sfr}\subset M_{\gamma,< \phi,\lambda}^{sfr}$$

Take the lattice $L_{\lambda}=\Z^{V(Q)}\times\Z,$ and extend the form $\chi$ onto $L_{\lambda}$ by the formula
$$\chi((\gamma,0),(0,1))=0,\,\, \chi((0,1),(\gamma,0))=-\sum\limits_{i\in I}\lambda^i\gamma^i,\,\,\chi((0,1),(0,1))=1.$$
We have the corresponding motivic quantum torus $\cT_{L_{\lambda}}^{mot}$ and its completion $\hat{\cT}_{L_{\lambda},C_{V_{\leq \phi}\times \R_{>0}}}^{mot}$ (resp. $\hat{\cT}_{L_{\lambda},C_{V_{<\phi}\times \R_{>0}}}^{mot}$).

\begin{prop}\label{conj_and_sfr}1) The action of $G_{\gamma}$ on $M_{\gamma,\leq\phi,\lambda}^{sfr}$
(resp. $M_{\gamma,<\phi,\lambda}^{sfr}$) is free, so that we have an algebraic variety
$$\cM_{\gamma,\leq\phi,\lambda}^{sfr}:=M_{\gamma,\leq\phi,\lambda}^{sfr}/G_{\gamma}\text{ (resp. } \cM_{\gamma,<\phi,\lambda}^{sfr}=M_{\gamma,<\phi,\lambda}^{sfr}/G_{\gamma}).$$
Put $$A_{\leq\phi,\lambda}^{sfr}=\sum\limits_{\gamma\in\Z_{\geq 0}^{V(Q)}}
[D(H^{\bullet,crit}(\cM_{\gamma,\leq\phi,\lambda}^{sp,sfr},W_{\gamma}))]\cdot T^{-\frac{\chi(\gamma,\gamma)}2}\hat{w}_{\gamma}\in\hat{\cT}_{Q,C_{V_{\leq\phi}}}^{mot},$$
and analogously for $A_{<\phi,\lambda}^{sfr}.$

2) We have the following identity in $\hat{\cT}_{L_{\lambda},C_{V_{\leq \phi}\times \R_{>0}}}^{mot}$:
\begin{equation}\label{conjugation}A_{V_{\leq\phi}}\hat{w}_{(0,1)}A_{V_{\leq\phi}}^{-1}=
\hat{w}_{(0,1)}\cdot A_{\leq\phi,\lambda}^{sfr},\end{equation}
and analogously for $A_{<\phi,\lambda}^{sfr}.$
\end{prop}

\begin{proof}1) is standard. We present its proof for completeness, for the case of $M_{\gamma,\leq\phi,\lambda}^{sfr}.$
The case $<\phi$ is analogous.

  Suppose that $g\in G_{\gamma},$ $g\ne 1,$ and $g\cdot(E,u)=(E,u)$ for some $(E,u)\in M_{\gamma,\leq\phi,\lambda}^{sfr}.$ Then we have a non-trivial automorphism $g:E\to E$ in $\rmMod \C Q,$ such that $g\circ u=u.$ Then $(g-\id)_{\mid \im(u)}=0,$ hence $(g-\id)$ defines a morphism $\coker(u)\to E.$ But $E\in M_{\gamma,V_{\leq\phi}},$ and $\coker(u)\in M_{\gamma',V_{>\phi}},$ hence $\Hom_{\C Q}(\coker(u),E)=0.$ Thus, $g=\id.$

2) We consider only the case $\leq\phi.$ The case $<\phi$ is analogous.

Consider the quiver $Q_{\lambda},$ such that $V(Q_{\lambda})=V(Q)\sqcup \{v\},$
the arrows between the vertices from $V(Q)$ are the same as in $Q,$ there are exactly $\lambda^i$ arrows from $i$ to $v,$ and no other arrows. Clearly, we have the natural identification
$$M_{(\gamma,1)}\cong\{(E\in M_{\gamma},u\in\Hom_{\C Q}(P_{\lambda},E))\}.$$

We define two kinds of extensions of the central charge $Z:\Z^{V(Q)}\to\C$ to $\Z^{V(Q)}\times\Z.$ First, choose $0<\alpha<\min_{i\in V(Q)}\Arg(Z(e_i)),$ where $\{e_i,i\in V(Q)\}$ is the standard basis. Put
$$\tilde{Z}_{1}((\gamma,0))=Z(\gamma),\quad \tilde{Z}_1((0,1))=\exp(\alpha \sqrt{-1}).$$
If $\widetilde{A}_V^{1}$ are DT series for $Q_{\lambda}$ with central charge $\tilde{Z}_1,$
then
$$\widetilde{A}_{V_{>\alpha}}^{1}=A_{\cH_+},\quad \widetilde{A}^1_{V_{\leq \alpha}}=
\sum\limits_{n\geq 0}\frac{T^{\frac{n^2}2}}{(1-T)\dots(1-T^n)}\hat{w}_{(0,n)}.$$
In particular, we have \begin{equation}\label{A_H_+}\widetilde{A}_{\cH_+}=A_{\cH_+}\cdot \sum\limits_{n\geq 0}\frac{T^{\frac{n^2}2}}{(1-T)\dots(1-T^n)}\hat{w}_{(0,n)}=A_{V_{>\phi}}A_{V_{\leq\phi}}\cdot\sum\limits_{n\geq 0}\frac{T^{\frac{n^2}2}}{(1-T)\dots(1-T^n)}\hat{w}_{(0,n)}.\end{equation}

Another choice of the central charge is the following. Choose some $0<\epsilon<\frac{\pi-\phi}2,$ and $t>0.$ Put
$$\tilde{Z}_{\epsilon,t}((\gamma,0))=Z(\gamma),\quad \tilde{Z}_{\epsilon,t}((0,1))=t\exp((\phi+\epsilon)\sqrt{-1}).$$
Again, we denote by $\widetilde{A}_V^{\epsilon,t}$ the DT series for $Q_{\lambda}$ with central charge $\tilde{Z}_{\epsilon,t}.$ First, we have
\begin{equation}\label{A_V_-}\widetilde{A}^{\epsilon,t}_{V_{>(\phi+2\epsilon)}}=A_{V_{>(\phi+2\epsilon)}}.\end{equation}

Second, for a fixed $\epsilon$ we have
\begin{equation}\label{A_V_+}\lim\limits_{t\to+\infty}\widetilde{A}^{\epsilon,t}_{V_{\leq\phi}}=A_{V_{\leq\phi}}.\end{equation} Finally, we claim that
\begin{equation}\label{middle_sector}\lim\limits_{\epsilon\to 0}\lim\limits_{t\to+\infty}\widetilde{A}_{\phi<\Arg(z)\leq \phi+2\epsilon}=
\sum\limits_{n\geq 0}\frac{T^{\frac{n^2}2}}{(1-T)\dots(1-T^n)}\hat{w}_{(0,n)}\cdot A^{sfr}_{\leq\phi,n\lambda}.\end{equation}

Indeed, take any $\gamma\in\Z_{\geq 0}^I,$ and $n\geq 0.$ Choose sufficiently small $\epsilon>0$ in such a way that for any non-zero dimension vector
$\gamma'\leq\gamma$ with $\Arg(Z(\gamma'))>\phi$ one has $\Arg(Z(\gamma'))>\phi+2\epsilon.$ Then, choose sufficiently large $t>0$ in such a way
that for any dimension vector $\gamma''\leq\gamma$ we have $$\phi<\Arg(t\exp((\phi+\epsilon)i)+Z(\gamma''))\leq \phi+2\epsilon.$$
Then, it is easy to see that for the central charge $\tilde{Z}_{\epsilon,t}$ we have the $G_{\gamma}$-equivariant identification
$$M_{(\gamma,n),\phi<\Arg(z)\leq \phi+2\epsilon}\cong M_{\gamma,\leq\phi,n\lambda}^{sfr}.$$
This implies the formula \eqref{middle_sector}.

Compairing \eqref{A_H_+} with \eqref{A_V_-}, \eqref{A_V_+} and \eqref{middle_sector}, and applying Theorem \ref{factorization_polynomial} we get the following chain of equalities:
\begin{multline*}A_{V_{>\phi}}A_{V_{\leq\phi}}\cdot\sum\limits_{n\geq 0}\frac{T^{\frac{n^2}2}}{(1-T)\dots(1-T^n)}\hat{w}_{(0,n)}=\tilde{A}_{\cH_+}^1=
\lim\limits_{\epsilon\to 0}\lim\limits_{t\to+\infty}\tilde{A}_{\cH_+}^{\epsilon,t}=\\
(\lim\limits_{\epsilon\to 0}\lim\limits_{t\to+\infty}\tilde{A}^{\epsilon,t}_{V_{>\phi+2\epsilon}})\cdot
(\lim\limits_{\epsilon\to 0}\lim\limits_{t\to+\infty}\tilde{A}^{\epsilon,t}_{\phi<\Arg(z)\leq \phi+2\epsilon})\cdot (\lim\limits_{\epsilon\to 0}\lim\limits_{t\to+\infty}\tilde{A}^{\epsilon,t}_{V_{\leq\phi}})=\\
A_{V_{>\phi}}\cdot(\sum\limits_{n\geq 0}\frac{T^{\frac{n^2}2}}{(1-T)\dots(1-T^n)}\hat{w}_{(0,n)}\cdot A^{sfr}_{\leq\phi,n\lambda})\cdot A_{V_{\leq\phi}}.\end{multline*}

Now, multiplying on the left by $A_{V_{>\phi}}^{-1},$ on the right by $A_{V_{\leq\phi}}^{-1},$ and compairing the coefficients for $\hat{w}_{(\gamma,1)},$ we obtain the desired formula \eqref{conjugation}.\end{proof}

\section{Categorification of quantum cluster algebras}
\label{s:CQCA}

In this section we mostly recall various notions and results related to categorification of (quantum) cluster algebras.
We also explain in Subsections \ref{ss:3CY} and \ref{ss:DTQFP} how the Donaldson-Thomas theory can be extended to some quivers with formal potentials.

\subsection{Ginzburg DG algebras}
\label{ss:GDGA}

The notion of Ginzburg DG algebra for a quiver with formal potential is due to Ginzburg, see \cite{G}.

Let $Q$ be any quiver and $W$ a formal potential on $Q,$ i.e. an infinite linear combination of cyclic paths of positive length.
The (complete) Ginzburg DG algebra $\hat{\Gamma}_{Q,W}$ is defined as follows. Define the graded quiver $\hat{Q}$ as follows. The vertex set of $\hat{Q}$ is the same as of $Q.$ Further, the edges of $\hat{Q}$ are:

1) The edges of $Q$ of degree zero;

2) The edges $a^*:j\to i$ for any arrow $a:i\to j$ in $Q,$ $\deg(a^*)=-1;$

3) For each vertex $i,$ the loop $t_i$ at $i,$ $\deg(t_i)=-2.$

As a graded algebra $\hat{\Gamma}_{Q,W}$ is the complete path algebra of $\hat{Q}$ (w.r.t. the length of paths). Further, the differential $d$ on $\hat{\Gamma}_{Q,W}$ is continuous, and on the arrows we have
$$d(a)=0,\, d(a^*)=\partial_aW\text{ for all edges }a\in E(Q),\quad d(t_i)=e_i(\sum\limits_{a}[a,a^*])e_i.$$

We have the derived category $D(\hat{\Gamma}_{Q,W}),$ and the subcategories
$$D^b(\hat{\Gamma}_{Q,W})\subset\Perf(\hat{\Gamma}_{Q,W})\subset D(\hat{\Gamma}_{Q,W}),$$
where $\Perf(\hat{\Gamma}_{Q,W})$ denotes the subcategory of perfect DG modules, and $D^b$ stands for DG modules with finite-dimensional total cohomology. Denote by $\iota$ the natural embedding
$$\iota:D^b(\hat{\Gamma}_{Q,W})\hookrightarrow \Perf(\hat{\Gamma}_{Q,W}),$$
see \cite{KY}.

Since the DG algebra $\hat{\Gamma}_{Q,W}$ is concentrated in non-positive degrees, we have the natural t-structures on $D(\hat{\Gamma}_{Q,W})$ (resp. $D^b(\hat{\Gamma}_{Q,W})$) with the heart being the category of modules (resp. finite-dimensional modules) over $H^0(\hat{\Gamma}_{Q,W}).$ Clearly, we have
$$H^0(\hat{\Gamma}_{Q,W})\cong \hat{J}_{Q,W}:=\widehat{\C Q}/\widehat{\langle\partial_aW,a\in E(Q)\rangle}.$$

We have isomorphisms $$K_0(\Perf(\hat{\Gamma}_{Q,W}))\cong\Z^{V(Q)},\quad K_0(D^b(\hat{\Gamma}_{Q,W}))\cong\Z^{V(Q)},$$
where the basis for $K_0(\Perf(\hat{\Gamma}_{Q,W}))$ is given by classes $[\hat{\Gamma}_{Q,W,i}]=[e_i\hat{\Gamma}_{Q,W}],$ $i\in V(Q),$ and the basis for $K_0(D^b(\hat{\Gamma}_{Q,W}))$ is given by the classes $[S_i],$ $i\in V(Q),$ where $S_i$
is the simple module at the vertex $i.$ Moreover, the Euler pairing
$$\chi:K_0(\Perf(\hat{\Gamma}_{Q,W}))\times K_0(D^b(\hat{\Gamma}_{Q,W}))\to\Z,\quad \chi([E],[F])=\sum\limits_{n\in\Z}\dim\Hom^n(E,F),$$
is perfect since
$$\chi([\hat{\Gamma}_{Q,W,i}],[S_j])=\delta_{ij}.$$

Also, the following is well-known, see \cite{G} or \cite{KY}.

\begin{prop}\label{Koszul_res}1) For $E,F\in D^b(\hat{\Gamma}_{Q,W})$ We have
$$\chi([\iota(E)],[F])=\chi_Q([E],[F])-\chi_Q([F],[E]).$$

2) We have the Koszul resolution
$$\{0\to \hat{\Gamma}_{Q,W,k}\to\bigoplus\limits_{\beta:k\to j}\hat{\Gamma}_{Q,W,j}\to \bigoplus\limits_{\alpha:i\to k}\hat{\Gamma}_{Q,W,i}\to \hat{\Gamma}_{Q,W,k}\to 0\}\cong \iota(S_k),$$
hence $$[\iota(S_k)]=\sum\limits_{i\in V(Q)}(a_{ki}-a_{ik})[\hat{\Gamma}_{Q,W,i}].$$
\end{prop}

\subsection{Mutations of quivers with formal potentials}
\label{ss:MQFP}

The notion of mutation for a quiver with formal potential is due to Derksen, Weyman and Zelevinsky \cite{DWZ}.

Let $Q$ be a finite quiver, and $W$ a formal potential on $Q.$ Take some vertex $k\in V(Q),$ and assume that there are no loops at $k.$ We recall the mutation $\mu_k(Q,W).$

First, the pre-mutation $\mu_k^{pre}(Q,W)$ is defined as follows. The set of vertices of $\mu_k^{pre}(Q)$ is the same as that
of $Q.$ Further, each edge $a:i\to k$ in $Q$ is replaced by the edge $\bar{a}:k\to i,$ and similarly for each edge $b:k\to j.$
Further, for each pair of edges $a:i\to k,$ $b:k\to j,$ we add a new edge $[ba]:i\to j.$

The potential $\mu_k^{pre}W$ on $\mu_k^{pre}Q$ is defined as the sum $W_1+W_2,$ where
$$W_1=\sum\limits_{\substack{a:i\to k,\\ b:k\to j}}[ba]\bar{a}\bar{b},$$
and $W_2$ is obtained from $W$ by replacing each occurrence of $ba$ (in the cyclic paths of $Q$) by $[ba],$
$b:k\to j,$ $a:i\to k.$

Now, two QP's $(Q,W)$ and $(Q',W'),$ with the same set of vertices $V(Q)=V(Q'),$ are called equivalent if there is
a continuous isomorphism $\psi:\widehat{\C Q}\to \widehat{\C Q'},$ preserving the idempotents in the vertices, such that
$\psi(W)=W'.$ The QP $(Q,W)$ is called trivial if its Jacobi algebra $\hat{J}_{Q,W}$ is zero. Further, the potential $W$ on $Q$ is called reduced if no loops  or $2$-cycles contribute to $W.$ The property to be trivial or reduced is preserved under equivalences.

\begin{defi}If $(Q_1,W_1)$ and $(Q_2,W_2)$ are QP's with $V(Q_1)=V(Q_2),$ then the direct sum
$$(Q,W)=(Q_1,W_1)\oplus (Q_2,W_2)$$
is defined by putting
$$V(Q):=V(Q_1)=V(Q_2),\quad E(Q):=E(Q_1)\sqcup E(Q_2),\quad W:=W_1\oplus W_2.$$\end{defi}

The following is proved in \cite{DWZ}.

\begin{prop}Any QP $(Q,W)$ is equivalent to the direct sum of the reduced one and the trivial one:
$$(Q,W)\sim (Q,W)_{red}\oplus (Q,W)_{triv}.$$
Both $(Q,W)_{red}$ and $(Q,W)_{triv}$ are determined uniquely up to equivalence, and their equivalence classes are determined by the equivalence class of $(Q,W).$\end{prop}

Now, one defines
$$\mu_k(Q,W)=\mu_k^{pre}(Q,W)_{red}.$$
Hence, mutation is well-defined up to equivalence.

\subsection{Decorated representations and their mutations}
\label{ss:DRM}

The notion of a decorated representations and the construction of their mutations appeared in \cite{DWZ}.

Let $(Q,W)$ be a formal QP. Put $R_Q=\C^{V(Q)},$ considered as semi-simple commutative algebra.

\begin{defi}A decorated representation $\cM=(M,V)$ of $(Q,W)$ is a pair of finite-dimensional  $\hat{J}_{Q,W}$-module $M$
and the finite-dimensional $R_Q$-module $V.$\end{defi}

If $(Q',W')$ is another QP with $V(Q')=V(Q),$ and $\cM'=(M',V')$ is a decorated representation of $(Q',W'),$ then $\cM$ and $\cM'$
are said to be equivalent if there exist a triple $(\psi,\phi,\eta),$ where:

1) $\psi:\widehat{\C Q}\to \widehat{\C Q'}$ is a continuos isomorphism of algebras, such that $\psi(W)=W'$

2) $\phi:M\to M'$ is an isomorphism of $\widehat{\C Q}$-modules, where the $\widehat{\C Q}$-module structure on $M'$ comes from $\psi;$

3) $\eta:V\to V'$ is an isomorphism of $R_Q$-modules.

Now, for any decorated representation $\cM=(M,V)$ of the QP $(Q,W),$ we can choose an equivalence $$(Q,W)\sim (Q,W)_{red}\oplus (Q,W)_{triv}$$ as above,
giving an isomorphism $\hat{J}_{Q,W}\cong\hat{J}_{(Q,W)_{red}},$ and treat $M$ as a $\hat{J}_{(Q,W)_{red}}$-module. Hence,
we get a decorated representation $\cM_{red}$ of $(Q,W)_{red},$ defined up to equivalence.

Now, assume that there are no loops at the vertex $k\in V(Q).$ We define the pre-mutation $\mu_k^{pre}(\cM)=(\bar{M},\bar{V}),$ which is a decorated representation
of $\mu_k^{pre}(Q,W).$ Put $$M_{in}:=\bigoplus\limits_{k\to j}M_j,\quad M_{out}:=\bigoplus\limits_{i\to k}M_i.$$ We have natural maps
$\alpha:M_{in}\to M_k,$ $\beta:M_k\to M_{out}.$ Further, we have a natural map
$$\gamma:M_{out}\to M_{in},$$
with the components $\partial_{[ba]}W_2,$ where $a:i\to k,$ $b:k\to j,$ and $W_2$ is the above component of the potential $\mu_k^{pre}(W),$
obtained from $W$ by replacing each peace $ba$ by the edge $[ba].$ Here we treat $\partial_{[ba]}W_2$ as elements of $\widehat{\C Q},$
replacing each arrow of the kind $[b'a']$ by the product $b'a'.$

A straightforward checking shows that $\alpha\gamma=0,$ $\gamma\beta=0.$ We define
$$\bar{M}_i:=M_i,\quad \bar{V}_i:=V_i\quad\text{for }i\ne k,$$ and
$$\bar{M}_k:=\frac{\ker \gamma}{\im\beta}\oplus\im\gamma\oplus\frac{\ker\alpha}{\im\gamma}\oplus V_k,\quad \bar{V}_k:=\frac{\ker\beta}{\ker\beta\cap\im\alpha}.$$

We need to define the action of arrows of $\mu_k^{pre}(Q)$ on $\bar{M}_i.$ If $c$ is an arrow of $Q,$ not incident to $k,$ then $c$ acts on $\bar{M}$ in the same way as on $M.$
further, if $a:i\to k,$ $b:k\to j$ are arrows in $Q,$ then $[ba]$ acts on $\bar{M}$ in the same way as the product $ba$ acts on $M.$
We need to define the action of arrows $\bar{a},$ $\bar{b}.$ Their action is given by the maps
$$\bar{\alpha}:M_{out}\to\bar{M}_k\quad\bar{\beta}:\bar{M}_k\to M_{in},$$
defined as follows. Choose some projection $\rho:M_{out}\to\ker\gamma,$ i.e. $\rho_{\mid\ker\gamma}=\id.$
Also, choose a splitting $\sigma:\ker\alpha/\im\gamma\to\ker\alpha.$

The non-zero components of $\bar{\alpha}$ are $\gamma:M_{out}\to\im\gamma,$ and $\pi\rho:M_{out}\to \ker\gamma/\im\beta,$
where $\pi:\ker\gamma\to\ker\gamma/\im\beta$ is the natural projection.

The non-zero components of $\bar{\beta}$ are the inclusion $\im\gamma\hookrightarrow M_{in},$ and $\iota\sigma:\ker\alpha\to\im\gamma/M_{in},$
where $\iota:\ker\alpha\to M_{in}$ is the inclusion.

One can check that the defined action of $\C \mu_k^{pre}(Q)$ on $\bar{M}$ actually defines the structure of $\hat{J}_{\mu_k^{pre}(Q,W)}$-module
structure on $\hat{M}.$ We put $$\mu_k(\cM)=\mu_k^{pre}(\cM)_{red}.$$

Again, the decorated representation $\mu_k(\cM)$ is defined up to equivalence.

\subsection{Derived equivalences}
\label{ss:DE}

The following is the result of Keller and Yang \cite{KY}:

\begin{theo}\label{Keller_Yang}1) If $$(Q,W)\sim (Q',W')\oplus (Q'',W''),$$
where $(Q'',W'')$ is trivial QP, then the Ginzburg DG algebras $\hat{\Gamma}_{Q,W}$ and $\hat{\Gamma}_{Q',W'}$
are quasi-isomorphic.

2) Let $\Gamma=\hat{\Gamma}_{Q,W},$ $\Gamma'=\hat{\Gamma}_{\mu_k(Q,W)}.$ We have two equivalences
$$\Phi_{k,+},\Phi_{k,-}:D(\Gamma)\to D(\Gamma'),$$
such that

$\bullet$ We have $\Phi_{k,\pm}^{-1}(\Gamma_i')=\Gamma_i,$ $i\ne k;$

$\bullet$ We have exact triangles

$$\Phi_{k,+}^{-1}(\Gamma_k')\to \bigoplus\limits_{i\to k}\Gamma_i\to \Gamma_k\to \Phi_{k,+}^{-1}(\Gamma_k')[1];$$
$$\Gamma_k\to \bigoplus\limits_{k\to j}\Gamma_j\to \Phi_{k,-}^{-1}(\Gamma_k')\to \Gamma_k[1].$$
\end{theo}

\subsection{Tilting}
\label{ss:Tilting}

Let $\cD$ be any triangulated category, and $(\cD_{\leq 0},\cD_{\geq 1})$ a t-structure on $\cD,$ with the heart $\cA=\cD_{\leq 0}\cap\cD_{\geq 0}.$ Let $(T,F)$ be a torsion pair in $\cA,$ i.e. $T,F\subset \cA$ are full subcategories,
$\Hom(T,F)=0,$ and for any object $E$ of $\cA$ there is a short exact sequence $$0\to E_{T}\to E\to E_{F}\to 0$$
with $E_{T}\in T,$ $E_{F}\in F.$ Then, the tilted t-structure $(\cD_{\leq 0}^{T[-1],F},\cD_{\geq 1}^{T[-1],F})$
is defined by the formulas
$$\cD_{\leq 0}^{(T[-1],F)}=\{X\in\cD_{\leq 1}\mid H^1_{\cA}(X)\in T,\},$$ $$\cD_{\geq 1}^{T[-1],F}=\{X\in\cD_{\geq 1}\mid H^1_{\cA}(X)\in F\}.$$
Also, put $\cA^{(T[-1],F)}:=\cD_{\leq 0}^{(T[-1],F)}\cap \cD_{\geq 0}^{(T[-1],F)}.$

Shifting by $1,$ we get the t-structure $(\cD_{\leq 0}^{(T,F[1])},\cD_{\geq 1}^{(T,F[1])}),$ with the heart $\cA^{(T,F[1])}.$

The following was proved by Keller and Yang \cite{KY}:

\begin{theo}With the above notation,
$$\Phi_{k,+}^{-1}(\rmMod \hat{J}_{\mu_k(Q,W)})=
(\rmMod \hat{J}_{Q,W})^{(S_{k}^{\oplus}[-1],S_k^{\perp})},$$
$$\Phi_{k,-}^{-1}(\rmMod \hat{J}_{\mu_k(Q,W)})=
(\rmMod \hat{J}_{Q,W})^{({}^{\perp}S_k,S_{k}^{\oplus}[1])}.$$
Here $S_k^{\oplus}$ is the subcategory of all direct sums of copies of $S_k.$\end{theo}

From this moment we assume that the quiver $Q$ does not have loops and $2$-cycles. We say that QP $(Q,W)$ is well-mutatable
at $k$ if the quiver $\mu_k(Q,W)$ does not contain $2$-cycles (it cannot contain loops).
If $\un{k}=(k_1,\dots,k_r)$ is the sequence of vertices of $Q,$ $k_i\ne k_{i-1},$ then the property to be well-mutatable
with respect to the sequence $\un{k}$ is defined inductively.

Let $\un{k}=(k_1,\dots,k_r)$ be the sequence of vertices, and suppose that the
QP $(Q,W)$ is well-mutatable with respect to the sequence $\un{k}.$ For $1\leq i\leq r,$ put $(Q_i,W_i)=\mu_{k_i}(\dots\mu_{k_1}(Q,W)).$  The following result was proved by K. Nagao \cite{N}.

\begin{theo}\label{Nagao}There exists a unique sequence of signs $\epsilon_1,\dots,\epsilon_r\in\{\pm\}$ such that
for each $1\leq i\leq r$ we have $$(\Phi_{k_i,\epsilon_i}\circ\dots\Phi_{k_1,\epsilon_1})^{-1}(\rmMod \hat{J}_{Q_i,W_i})=(\rmMod \hat{J}_{Q,W})^{(T_i[-1],F_i)},$$ for some torsion pair $(T_i,F_i)$ in $\rmMod \hat{J}_{Q,W}.$\end{theo}

\begin{proof}We present the proof here for completeness. For convenience, for two strictly full subcategories $\cC_1,\cC_2$ of an abelian or triangulated category,
denote by $\cC_1\star\cC_2$ the subcategory of all extensions of the objects of $\cC_2$ by the objects of $\cC_1.$

\begin{lemma}Let $\cD$ be a triangulated category, and $(\cD_{\leq 0},\cD_{\geq 1})$ a t-structure with the heart $\cA.$ Let $(T,F)$ be a torsion pair in $\cA,$ and $(T',F')$ a torsion pair in $\cA^{(T[-1],F)}$ such that
$T'\subset F[0]$ (resp. $F'\subset T[-1]$). Then we have
$$(\cD_{\leq 0}^{(T[-1],F)})^{(T'[-1],F')}=\cD_{\leq 0}^{(T''[-1],F'')}\text{ (resp. }(\cD_{\leq 0}^{(T[-1],F)})^{(T',F'[1])}=\cD_{\leq 0}^{(T''[-1],F'')})$$
(and hence the same for $\cD_{\geq 1}$ and $\cA$), where $T''=T\star T',$ $F''=F\cap T'^{\perp}.$ (resp. $T''=T\cap{}^{\perp}(F'[1]),$ $F''=(F'[1])\star F$).\end{lemma}

\begin{proof}Straightforward.\end{proof}

Now define the sequence $\epsilon_1,\dots,\epsilon_r$ inductively. The unique choice for $\epsilon_1$ is $"+"$.
If the signs $\epsilon_1,\dots,\epsilon_i$ are already defined, for some $1\leq i\leq r-1,$ then consider two cases.

1) The simple $\hat{J}_{Q_i,W_i}$-module $S_{k_{i+1}}$ is contained in $\Phi_{k_i,\epsilon_i}\circ\dots\Phi_{k_1,\epsilon_1}(F_i).$
Then we put $\epsilon_{i+1}=+.$

2) The simple $\hat{J}_{Q_i,W_i}$-module $S_{k_{i+1}}$ is contained in $\Phi_{k_i,\epsilon_i}\circ\dots\Phi_{k_1,\epsilon_1}(T_i[-1]).$ Then we put $\epsilon_{i+1}=-.$

It follows from the above Lemma that the constructed sequence satisfies the required properties. It is easy to see that it is unique with the required properties.
\end{proof}

\subsection{Connection with Plamondon's results}
\label{ss:Plamondon}

Plamondon \cite{P1} constructed cluster character for arbitrary quiver with potential.
For any triangulated category $\cD,$ and any rigid object $E\in\cD$ (i.e. $\Hom^1(E,E)=0$), he defines the subcategory
$$\pr_{\cD}E:=\{Cone(E^{1}\to E^0)\mid E^0,E^1\in\Add(E)\},$$
where $\Add(E)$ is the full subcategory of finite direct sums of direct summands of $E.$ Further,
he proves that the natural projection
$$\pi:\Perf(\hat{\Gamma}_{Q,W})\to\cC:=\Perf(\hat{\Gamma}_{Q,W})/D^b(\hat{\Gamma}_{Q,W})$$
yields an equivalence $$\pr_{\Perf(\hat{\Gamma}_{Q,W})}(\hat{\Gamma}_{Q,W})\stackrel{\sim}{\to}\pr_{\cC}(\pi(\hat{\Gamma}_{Q,W})).$$
Further, Plamondon defines a suitable subcategory of the later category, and constructs the so-called cluster character,
describing all cluster monomials via Euler characteristics of quiver Grassmannians. For the details, we refer the reader to \cite{P1}.

We will need the following result which is analogous to Plamondon's Theorem (\cite{P1}, Theorem 2.18).

\begin{theo}\label{Plamondon}Suppose that the QP $(Q,W)$ is well-mutatable with respect to the sequence $\un{k}=(k_1,\dots,k_r),$
and $$(Q_r,W_r)=\mu_{k_r}(\dots(\mu_{k_1}(Q,W))\dots).$$
Then there is a sequence of signs $\epsilon_1,\dots,\epsilon_r$ such that for $1\leq i\leq r$ such that
$$(\Phi_{k_r,\epsilon_r}\circ\dots\circ\Phi_{k_1,\epsilon_1})^{-1}(\hat{\Gamma}_{Q_r,W_r})\in\pr_{\Perf(\hat{\Gamma}_{Q,W})}
(\hat{\Gamma}_{Q,W}[-1]).$$\end{theo}

The next Lemma shows that this sequence of signs is actually the same as in Theorem \ref{Nagao}.

\begin{lemma}\label{Nagao_Plamondon}In the assumptions of the Theorem \ref{Plamondon}, the sequence of signs is actually unique
and satisfies the property of Theorem \ref{Nagao}.\end{lemma}

\begin{proof}For convenience, put $$\hat{\Gamma}:=\hat{\Gamma}_{Q,W},\quad\hat{\Gamma}(i):=\hat{\Gamma}_{Q_i,W_i},\quad (Q_i,W_i)=\mu_{k_i}(\dots(\mu_{k_1}(Q,W))\dots),\quad 1\leq i\leq r.$$

The statement of the lemma essentially follows from the fact that
\begin{equation}\label{descr_of_pr}\pr_{\Perf(\hat{\Gamma})}(\hat{\Gamma}[-1])=D(\hat{\Gamma})_{\leq 1}\cap {}^{\perp}D(\hat{\Gamma})_{\leq(-1)}\cap\Perf(\hat{\Gamma}),\end{equation}
which is proved in \cite{P1}.

Now, let $\epsilon_1,\dots,\epsilon_r$ be some sequence of signs from Theorem \ref{Plamondon}, and
$\epsilon_1',\dots,\epsilon_r'$ the unique sequence from Theorem \ref{Nagao}. Suppose that for some $0\leq i<r$ we have
$$\epsilon_j=\epsilon_j',\,1\leq j\leq i,\quad \epsilon_{i+1}\ne\epsilon_{i+1}'.$$
Put $$\Phi(i+1):=\Phi_{k_{i+1},\epsilon_{i+1}}\dots\Phi_{k_1,\epsilon_1}.$$
It follows from the proof of Theorem \ref{Nagao} that one of the following holds:

(1) $\Phi(i+1)^{-1}(S_{k_{i+1}})\in (\rmMod \hat{J}_{Q,W})[-2];$

(2) $\Phi(i+1)^{-1}(S_{k_{i+1}})\in (\rmMod \hat{J}_{Q,W})[1].$

But we have that
$$\Hom_{D(\hat{\Gamma})}(\Phi(i+1)^{-1}(\hat{\Gamma}(i+1)),\Phi(i+1)^{-1}(S_{k_{i+1}}))=\Hom_{D(\hat{\Gamma}(i+1))}(\hat{\Gamma}(i+1),S_{k_{i+1}})\ne 0,$$
which contradicts to $\Phi(i+1)^{-1}(\hat{\Gamma}(i+1))\in\pr_{\Perf(\hat{\Gamma})}(\hat{\Gamma}[-1]),$ by \eqref{descr_of_pr}.

Hence, we have $\epsilon_j=\epsilon_j',$ $1\leq j\leq r.$
\end{proof}

Taking the signs $\epsilon_i$ as in Theorem \ref{Nagao}, put
 $$\Phi(r):=\Phi_{k_{r},\epsilon_{r}}\dots\Phi_{k_1,\epsilon_1}.$$ Combining Lemma \ref{Nagao_Plamondon} and the results of \cite{P2}, we get the following Corollary.

\begin{cor}\label{fin_dim_of_H^1}In the assumptions of Theorem \ref{Nagao}, each representation
$H^1(\Phi(r)^{-1}(\hat{\Gamma}_{Q_r,W_r,j})),$ $j\in V(Q),$ is either zero, or the decorated representation $(H^1(\Phi(r)^{-1}(\hat{\Gamma}_{Q_r,W_r,j})),0)$
of the QP $(Q,W)$ is obtained by the inverse sequence of mutations from the trivial decorated representation $(0,e_jR_{Q_r})$ of the QP $(Q_r,W_r).$ In particular, we have
$$\dim H^1(\Phi(r)^{-1}(\hat{\Gamma}))<\infty.$$\end{cor}

\subsection{$3$CY $A_{\infty}$-categories}
\label{ss:3CY}

We refer to \cite{KS08}, \cite{Kaj}, \cite{C}, \cite{CL1}, \cite{CL2} for cyclic $A_{\infty}$-categories and potentials.

First, we recall the definition of an $A_{\infty}$-category with scalar product. Below we use Quillen notation $\pm$ meaning that the signs come from Koszul sign rule.

\begin{defi}A $\C$-linear $A_{\infty}$-category $\cC$ with finite-dimensional graded spaces of morphisms is called an $A_{\infty}$-category with scalar product
of degree $d$ if
there are fixed perfect pairings $$\langle,\rangle:\Hom^i(X,Y)\otimes\Hom^{d-i}(Y,X)\to\C,$$
which are super-symmetric, and
$$\langle m_n(\alpha_0,\dots,\alpha_{n-1}),a_n\rangle=\pm \langle m_n(a_1,\dots,a_n),a_0\rangle$$
for homogeneous $\alpha_0,\dots,\alpha_n$ with
$$\deg(\alpha_0)+\dots+\deg(\alpha_n)=n+d-2.$$
We call such $A_{\infty}$-categories $d$-dimensional Calabi-Yau (or just $d$CY, or cyclic).

Further, an $A_{\infty}$-functor $F:\cC_1\to\cC_2$ between $d$CY $A_{\infty}$-categories is said to be compatible with $d$CY structure (or just cyclic)
if
$$\langle f_1(\alpha),f_1(\beta)\rangle=\langle \alpha,\beta\rangle,\quad \deg(\alpha)+\deg(\beta)=d;$$
$$\sum\limits_{i=0}^{n-1} \langle f_{i+1}(\alpha_0,\dots,\alpha_i), f_{n-i}(\alpha_{i+1},\dots,\alpha_n),\rangle$$
for homogeneous $\alpha_0,\dots,\alpha_n,$ $n\geq 2.$
\end{defi}

Recall the definition of the $A_{\infty}$-category $Tw\, \cC$ for an $A_{\infty}$-category $\cC.$
First, take the $A_{\infty}$-category $\cC'$ obtained from $\cC$ by adding all the shifts.
Now, let $S$ be some sequence of objects $(X_1,\dots,X_n)$ in $\cC'.$ We have
the $A_{\infty}$-algebra
$$\End_+(S)=\bigoplus\limits_{1\leq i,j\leq n}\Hom(X_i,X_j),$$
and its nilpotent $A_{\infty}$-subalgebra
$$\End_+(S)=\bigoplus\limits_{1\leq i<j\leq n}\Hom(X_i,X_j).$$

Then the twisted complex is a pair $(S,\alpha),$ where $\alpha\in\End_+(S)^1$ is the solution of the MC equation
$$\sum\limits_{n=1}^{\infty}(-1)^{\frac{n(n+1)}{2}}m_n(\alpha,\dots,\alpha)=0.$$
For a pair $((S_1,\alpha_1),(S_2,\alpha_2))$ of objects in $Tw\,\cC$
put
\begin{equation}\Hom((S_1,\alpha_1),(S_2,\alpha_2))=\bigoplus\limits_{X\in S_1,Y\in S_2}\Hom(X,Y).\end{equation}
For a sequence,
$((S_0,\alpha_0),\dots,(S_n,\alpha_n))$ of objects in
$Tw\,\cC$ and homogeneous morphisms $x_i\in
\Hom((S_{i-1},\alpha_{i-1}),(S_i,\alpha_i))$ we put
\begin{equation}m_n(x_n,\dots,x_1)=\sum\limits_{k_0,\dots,k_n\geq 0}(-1)^{\epsilon}
m_{n+k_0+\dots+k_n}(\alpha_n^{k_n},x_n,\alpha_{n-1}^{k_{n-1}},\dots,x_1,\alpha_0^{k_0}),\end{equation}
where $\epsilon=\sum\limits_{n\geq i>j\geq
0}(\deg(x_i)+k_i)k_j+\sum\limits_{i=0}^n
\frac{k_i(k_i+1)}2+\sum\limits_{i=1}^n ik_i.$

If $\cC$ is a $d$CY $A_{\infty}$-category, then so is $Tw\,\cC,$ with the obvious scalar product.

Let us take any formal QP $(Q,W).$ Then one can associate to it a $3$CY $A_{\infty}$-category $\cC_{Q,W}$
with $Ob(\cC_{Q,W})=V(Q),$ such that we have
$$\Ho(Tw\, \cC_{Q,W})\cong D^b(\hat{\Gamma}_{Q,W}).$$

Namely, we denote the objects of $\cC_{Q,W}$ simply by $S_i,$ $i\in V(Q),$ and put
$$\Hom^0(S_i,S_i)=\Hom^3(S_i,S_i)=\C,\quad\Hom^1(S_i,S_j)=\C^{a_{ji}}=\Hom^2(S_j,S_i),\quad i,j\in V(Q);$$
$$\Hom^{<1}(S_i,S_j)=0=\Hom^{>2}(S_i,S_j),\quad i\ne j,\quad\Hom^{<0}(S_i,S_i)=0=\Hom^{>3}(S_i,S_i).$$

The higher products and the pairing are the following. First, units $1\in\C=\Hom^0(S_i,S_i)$ are strict identity morphisms. Second,
the products $$m_2:\Hom^1(S_i,S_j)\otimes\Hom^2(S_j,S_i)\to\Hom^3(S_j,S_j)=\C,$$ $$m_2:\Hom^2(S_i,S_j)\otimes\Hom^1(S_j,S_i)\to\Hom^3(S_j,S_j)=\C$$
are just standard perfect pairings. They define the pairing $\langle,\rangle$ of degree $3$ on $\cC_{Q,W}.$ Further,
the higher products
$$m_n:\Hom^1(S_{i_{n-1}},S_{i_n})\otimes\dots\otimes\Hom^1(S_{i_0},S_{i_1})\to\Hom^2(S_{i_0},S_{i_n})$$
are given by the elements of the spaces
$$\Hom^1(S_{i_n},S_{i_0})^{\vee}\otimes\Hom^1(S_{i_{n-1}},S_{i_n})^{\vee}\otimes\dots\otimes\Hom^1(S_{i_0},S_{i_1})^{\vee},$$
which are the components of the potential $W.$ All the other higher products are zero.

The following is well-known, see \cite{KS08} and \cite{KY}

\begin{prop}\label{3CYlifting}We have natural equivalences of triangulated categories
$$\Ho(Tw\, \cC_{Q,W})\cong\Perf(\cC_{Q,W})\cong D^b(\hat{\Gamma}_{Q,W}),\quad D^b(\cC_{Q,W})\cong \Perf(\hat{\Gamma}_{Q,W}).$$
The equivalences $\Phi_{k,\pm}$ of Theorem \ref{Keller_Yang} are induced by $A_{\infty}$-functors
$$\phi_{k,\pm}:\cC_{Q,W}\to Tw\,\cC_{\mu_k(Q,W)},$$
compatible with the $3$CY structures.\end{prop}

If $A$ is a $3$CY algebra, then we have a formal power series on $A^1:$
$$W_A(\alpha)=\sum\limits_{n=1}^N\frac{\langle m_n(\alpha,\dots,\alpha),\alpha\rangle}{n+1}.$$
If $f:A\to B$ is an $A_{\infty}$-morphism compatible with $3$CY structures, then
we have formal map $\hat{f}:A\to B,$
$$\hat{f}(\alpha)=\sum\limits_{n\geq 1}f_n(\alpha,\dots,\alpha),$$
and we have
$$\hat{f}^*(W_B)=W_A.$$

We would like to point out the relation of this potential with the potentials on the moduli of quiver representations. Namely, let $(Q,W)$
be a polynomial QP
take the Calabi-Yau $A_{\infty}$-category $\cC_{Q,W}$ defined above, and take the twisted complex $S_{\gamma}:=\bigoplus\limits_{i\in V(Q)}S_i^{\oplus\gamma^i}.$ Then we have natural identification $$\End^1(S_{\gamma})\cong M_{\gamma},$$
and the potential $W_{\End(S_{\gamma})}$ corresponds to the potential $W_{\gamma}.$ More generally, if $W$ is a formal potential,
then we also have such an identification, but $W_{\gamma}$ is a function on the formal neighborhood of $Nilp_{\gamma}\subset M_{\gamma}.$
Moreover, if $\alpha\in\End^1(S_{\gamma})$ is an MC solution corresponding to the representation of $\hat{J}_{Q,W},$ then the potential
on $\End^1(S_{\gamma},\alpha)$ is just given by the formula $$W_{\alpha}(z)=W_{\gamma}(\alpha+z).$$

We will need the following result of \cite{Kaj}.

\begin{theo}\label{Kajiura}For any $3$CY $A_{\infty}$-algebra, there is a $3$CY $A_{\infty}$-isomorphism $$A\cong A_{min}\oplus A_{triv}$$
where $A_{min}$ is minimal (i.e. $m_1=0$) and $A_{triv}$ is trivial (i.e. $H^{\bullet}(A_{triv},m_1)=0,$ $m_{>1}=0$) $3$CY $A_{\infty}$-algebras.
Moreover, both $A_{min}$ and $A_{triv}$ are defined uniquely up to a cyclic $A_{\infty}$-automorphism.
In particular, we have that after a formal change of coordinates, $W_A$ has the form
$$W^{min}\oplus Q_A\oplus N_A,$$ where $W^{min}$ is the potential on $A_{min}^1\cong H^1(A,m_1),$ and $Q_A$ is the quadratic form on the space
$A^1/\ker(m_1)$ given by the formula
$$Q_A(\alpha,\alpha)=\frac12\langle m_1(\alpha),\alpha\rangle,$$
and $N_A$ is the zero function on $\im(m_1:A^0\to A^1).$\end{theo}

\subsection{DT theory for a quiver with formal potential}
\label{ss:DTQFP}

Take some formal QP $(Q,W).$ In general (if $W$ is not polynomial) we do not have
regular functions $W_{\gamma}$ on $M_{\gamma}.$ However, we have well-defined functions $W_{\gamma}$ on the formal neighborhoods
 of $Nilp_{\gamma}\subset M_{\gamma}.$ Suppose that we have some central charge $Z$ on $Q.$
Then for any sector $V\subset\cH_+$ we have a well-defined $G_{\gamma}$-equivariant closed subset $M_{\gamma,V}^{sp}\subset M_{\gamma,V},$ which
parameterizes representations of $\hat{J}_{Q,W}.$ Define the extension-closed subcategory $\cC_V\subset\rmmod\hat{J_{Q,W}}\subset D^b(\hat{\Gamma}_{Q,W})$
as a full subcategory consisting of isomorphism classes of representations in $M_{\gamma,V}^{sp},$ $\gamma\in\Z_{\geq 0}^{V(Q)}.$

Note that in general we do not have well-developed DT theory for a quiver with formal potential (see the discussion in \cite{KS}, Subsection 7.1).  However,
in Appendix we will prove the following result.

\begin{theo}\label{DT_for_formal}Suppose that for some formal QP $(Q,W)$ and polynomial QP $(Q',W')$ we have a cyclic $A_{\infty}$-functor
$\phi:\cC_{Q,W}\to Tw\,\cC_{Q',W'},$ inducing an equivalence
$\Phi:D^b(\hat{\Gamma}_{Q,W})\stackrel{\sim}{\to} D^b(\hat{\Gamma}_{Q',W'}).$ Let $Z$ be a central charge on $Q.$ Then one can define the classes
$$[D(H^{\bullet,crit}_{c,G_{\gamma}}(M_{\gamma,V})^{sp},W_{\gamma})]\in\hat{R},$$
for all sectors $V\subset \cH_+,$ and $\gamma\in\Z_{\geq 0}^{V(Q)},$ such that the following holds.

1) Suppose that there is only one $G_{\gamma}$-orbit in $M_{\gamma,V}^{sp},$ and for the corresponding representation $E$
of $\hat{J}_{Q,W}$ we have $$\Ext^1(E,E)=0,\quad \chi_Q(\gamma,\gamma)\equiv \dim\Ext^0(E,E)\text{ mod }2.$$
Then we have
\begin{equation}\label{crit_cohom_formal}[D(H^{\bullet,crit}_{c,G_{\gamma}}(M_{\gamma,V})^{sp},W_{\gamma})]=[H^{\bullet}(\mrB\Aut(E))]\cdot T^{\dim_{\C}\Aut(E)}.\end{equation}

2) Define the DT series $A_V$ using the classes $[D(H^{\bullet,crit}_{c,G_{\gamma}}(M_{\gamma,V})^{sp},W_{\gamma})]\in\hat{R},$
as in the formula \eqref{formula_for_A_V}.

Suppose that we have a central charge $Z'$ on $Q',$ and for some sectors $V,V'\subset\cH_+$ we have that
$$\Phi(\cC_V)=\cC_{V'},\quad \cC_V\subset D^b(\hat{\Gamma}_{Q,W}),\,\cC_{V'}\subset D^b(\hat{\Gamma}_{Q',W'}),$$ and assume that
\begin{equation}\label{parity_preserved}\chi_{Q'}([\Phi](\gamma),[\Phi](\gamma))\equiv \chi_Q(\gamma,\gamma)\text{ mod }2,\quad\gamma\in\Z^{V(Q)}.\end{equation}
Then we have
$$A_{V'}=[\Phi](A_V),$$
where in the last formula $[\Phi]$ denotes the induced map on completions of motivic quantum tori.

3) If the sector $V$ is the disjoint union of two sectors $V_1\sqcup V_2$ (in the clockwise order), then we have factorization:
$$A_V=A_{V_1}A_{V_2}.$$
\end{theo}

\begin{remark}\label{parity_for_mutations}It is well-known (and is easy to check) that the assumption \eqref{parity_preserved}
holds for equivalences $\Phi$ coming from a single mutation, and hence from the sequence of mutations.\end{remark}

\subsection{Torsion pairs and stability}
\label{ss:TPStab}

If $\cH_+=V_-\sqcup V_+$ is the decomposition into the disjoint union of two sectors (in the clockwise order). Then we have a torsion pair $(\cC_{V_-},\cC_{V_+})$ in the category $\rmmod \hat{J}_{Q,W}.$ In the paper \cite{N} it is shown that the torsion pairs on $\rmmod \hat{J}_{Q,W}$ which appear in Theorem \ref{Nagao}, are actually obtained in this way.

Now let $\un{k}=(k_1,\dots,k_r)$ be a sequence of vertices, and $\epsilon_1,\dots,\epsilon_r$ a sequence of signs
from Theorem \ref{Nagao}. Again, put $$(Q_i,W_i)=\mu_{k_i}(\dots(\mu_{k_1}(Q,W))\dots),\quad  \Phi(i):=\Phi_{k_i,\epsilon_i}\circ\dots\Phi_{k_1,\epsilon_1},\quad 0\leq i\leq r.$$
and
$$S(i):=\begin{cases}\Phi(i-1)^{-1}(S_{k_i}) & \text{if }\Phi(i-1)^{-1}(S_{k_i})\in F_{i-1};\\
\Phi(i-1)^{-1}(S_{k_i})[1] & \text{if }\Phi(i-1)^{-1}(S_{k_i})\in T_{i-1}[-1].\end{cases}$$

By Theorem \ref{DT_for_formal} and Proposition \ref{3CYlifting}, if the potential $W_r$ is polynomial, then we have DT theory for
QP $(Q,W).$

\begin{theo}\label{torsion_pairs_via_Z}
1)With the above notation, there exist central charges $Z_i$ on $Q,$ and angles $0<\phi_i<\pi,$ such that $$\cC_{V_{>\phi_i}}=T_i\cap(\rmmod \hat{J}_{Q,W}),\quad \cC_{V_{<\phi_i}}=F_i\cap(\rmmod \hat{J}_{Q,W}).$$

2) Assume that the potential $W_r$ is polynomial. Then, we have
$$A^{Z_i}_{V_{>\phi_i}}=(-T^{\frac12}\hat{w}_{[S(1)]};T)_{\infty}^{\epsilon_1}\dots (-T^{\frac 12}\hat{w}_{[S(i)]};T)_{\infty}^{\epsilon_i},$$
where $(z;q)_{\infty}=\prod\limits_{n=0}^{\infty}(1-q^nz)$ is the $q$-Pochammer symbol.
\end{theo}

\begin{proof}1) This result is proved in \cite{N}, however, we propose a different and simpler proof. First we note that the required properties for $Z_i$ and $\phi_i$ are equivalent to the following:
for all non-zero $E\in T_i\cap(\rmmod \hat{J}_{Q,W})$ (resp. $E\in F_i\cap(\rmmod \hat{J}_{Q,W})$) we have $\Arg(Z_i([E]))>\phi_i$ (resp. $\Arg(Z_i([E]))<\phi_i$).
Further, this is equivalent to the inequality
$$\im(\exp((\pi-\phi_i)\sqrt{-1})Z_i([\Phi(i)^{-1}(F)]))>0$$
for any non-zero $F\in\rmmod \hat{J}_{Q_i,W_i}.$ Since any $F\in \hat{J}_{Q_i,W_i}$ is an iterative extension of $S_j$'s, we just need the inequalities
$$\im(\exp((\pi-\phi_i)\sqrt{-1})Z_i([\Phi(i)^{-1}(S_j)]))>0,\quad j\in V(Q_i)=V(Q).$$
Put $$\phi_i:=\frac{\pi}2,\quad\im(Z_i(e_j)):=1,\quad\re(Z_i([\Phi(i)^{-1}(S_j)])):=1,\quad j\in V(Q)=V(Q_i).$$
This determines $Z_i$ uniquely, and it satisfies the required properties.

It is also clear
that for any fixed $\phi_i,$ the set of central charges $Z_i,$ satisfying the required properties, is in bijection
with the set $$\R_{>0}^{V(Q)}\times\R_{>0}^{V(Q_i)}.$$

2) First, the series $A^{Z_i}_{V_{>\phi_i}}$ do not depend on $Z_i$ and $\phi_i$ (satisfying the required properties) so we can choose $Z_i$ and $\phi_i$ as we want.
We proceed by induction on $i.$ For $i=1,$ we have that all non-zero representations in $M^{Z_1}_{\gamma,V_{>\phi_i}}$
are finite direct sums of copies of $S_{k_1}=S(1),$ hence by Theorem \ref{DT_for_formal}, 1), and Remark \ref{parity_for_mutations} we have
$$A^{Z_1}_{V_{>\phi_i}}=\sum\limits_{n\geq 0}[H^{\bullet}(BGL(n))]T^{\frac{n^2}2}\hat{w}_{n[S(1)]}=(-T^{\frac12}\hat{w}_{[S(1)]};T).$$
We are left to prove that \begin{equation}\label{inductive}A^{Z_i}_{V_{>\phi_i}}=A^{Z_{i-1}}_{V_{>\phi_{i-1}}}\cdot (-T^{\frac12}\hat{w}_{[S(i)]};T)^{\epsilon_i},\quad 1<i\leq r.\end{equation}

We will consider the case $\epsilon_i=+$ (the case $\epsilon_i=-$ is analogous). Then we have $S(i)=\Phi(i-1)^{-1}(S_{k_i}).$ It is clear from the proof of 1) that we
can choose $Z_{i-1}$ and $\phi_{i-1}$ in such a way that $$\Arg(Z_{i-1}([S(i)])\exp((\pi-\phi_{i-1})\sqrt{-1}))>\Arg(Z_{i-1}([\Phi(i-1)^{-1}(S_j)])\exp((\pi-\phi_{i-1})\sqrt{-1}))$$
for $j\ne k_i.$ Then, for sufficiently small $\delta>0,$ the pair $(Z_i=Z_{i-1},\phi_i=\Arg(Z_{i-1}([S(i)]))-\delta)$ satisfies
the required properties. We have (again by Theorem \ref{DT_for_formal}, 1), and Remark \ref{parity_for_mutations})
$$A^{Z_i}_{\phi_{i}<\Arg(z)\leq\phi_{i-1}}=\sum\limits_{n\geq 0}[H^{\bullet}(BGL(n))]T^{\frac{n^2}2}\hat{w}_{n[S(i)]}=(-T^{\frac12}\hat{w}_{[S(i)]};T),$$
which (together with Theorem \ref{DT_for_formal}, 3)) implies \eqref{inductive}. This proves 2).
\end{proof}

\section{Laurent phenomenon via stable framed representations}
\label{s:LPSFR}

Let now $L=\Z^m,$ $\Lambda:L\times L\to\Z$ a skew-symmetric form, and $$\cT_{\Lambda}^{mot}:=\cT_{\Lambda}\otimes_{\Z[q^{\pm\frac12}]}\hat{R}[T^{\frac12}],\quad q^{\frac12}\mapsto T^{\frac12},$$ the associated
motivic quantum torus. Fix some $1\leq n\leq m,$ and suppose that
$$(M(c)=X^c,\tilde{B})$$ is the skew-symmetric quantum seed. Let $S$ be its mutation-equivalence class,
and $\cA_S\subset\cF_{\Lambda}$ the associated quantum cluster algebra.

If $(M',\tilde{B}')$ is some quantum seed, mutation-equivalent to $(M,\tilde{B}),$ then by Laurent phenomenon,
for each $\lambda\in\Z_{\geq 0}^m$ we have that $M'(\lambda)\in\cT_{\Lambda},$ so we have a well-defined element
$$M^{'mot}(\Lambda):=M'(\lambda)\otimes 1\in\cT_{\Lambda}^{mot}.$$

Take the quiver $Q$ as in Subsection \ref{ss:MSSQS}. Namely, $V(Q)=\{1,\dots,m\},$ and $Q$ is without loops and 2-cycles, so that
$$a_{jj}-a_{ij}=b_{ij},\quad 1\leq i\leq m,\,1\leq j\leq n.$$
We do not impose any restrictions for the numbers $a_{ij}$ for $n+1\leq i,j\leq m.$

Let $W$ be some formal potential on $Q.$
We have an identification $K_0(\Perf(\hat{\Gamma}_{Q,W}))\cong\Z^m,$ and for the embedding
$\iota:D^b(\hat{\Gamma}_{\tilde{Q},W})\hookrightarrow\Perf(\hat{\Gamma}_{\tilde{Q},W})$
we have $$\iota([S_j])=(b_{1j},\dots,b_{mj}),\quad 1\leq j\leq n,$$
by Proposition \ref{Koszul_res}.
It follows from the compatibility condition on $\Lambda$ and $\tilde{B}$ that
\begin{equation}\label{compatibility}\Lambda([\hat{\Gamma}_{Q,W,i}],[\iota(S_j)])=-\chi([\hat{\Gamma}_{Q,W,i}],[\iota(S_j)])=-\delta_{ij},\quad 1\leq i\leq m,\,1\leq j\leq n.\end{equation}

Let $\cC_{[1,\dots,n]}\subset \rmmod \hat{J}_{Q,W}$ be the subcategory
of representations supported on the vertices $1,\dots,n.$ It follows from \eqref{compatibility} and Proposition \ref{Koszul_res} that for any two objects
$E,F\in\cC_{[1,\dots,n]}$ we have
$$\Lambda([\iota(E)],[\iota(F)])=-\chi_Q([E],[F])+\chi_Q([F],[E]).$$ Therefore,
the coordinate sublattice $\Z^n=K_0(\cC_{[1,\dots,n]})\subset K_0(\rmmod \hat{J}_{Q,W})=\Z^m$
defines the motivic quantum torus
$\cT_{[1,\dots,n]}^{mot},$ equipped with the natural injective morphism
$$\cT_{[1,\dots,n]}^{mot}\hookrightarrow \cT_{\Lambda}^{mot},\quad ,\hat{w}_{[E]}\mapsto X^{[\iota(E)]}.$$
From now on, we identify the elements of $\cT_{[1,\dots,n]}^{mot}$ with their images in $\cT_{\Lambda}^{mot},$
and similarly for various completions.

Now, let $\un{k}=(k_1,\dots,k_r)$ be any sequence of vertices, $1\leq k_i\leq n,$ $k_i\ne k_{i+1}.$ We assume that $W$ is well-mutatable with
 respect to $\un{k}.$
As above, we put
$$(Q_i,W_i)=\mu_{k_i}(\dots(\mu_{k_1}(\tilde{Q},W))\dots),\quad 1\leq i\leq r.$$

{\noindent{\bf Assumption. }}We may and will assume that $W_r$ is a polynomial potential on $Q_r.$

We have a sequence of signs $\epsilon_1,\dots,\epsilon_r\in\{\pm\}$ as in Theorem \ref{Nagao}.
Again, we put $$\Phi(i)=\Phi_{k_i,\epsilon_i}\circ\dots\circ\Phi_{k_1,\epsilon_1},\quad 0\leq i\leq r,$$
so that
$$\Phi(i)^{-1}(\rmMod \hat{J}_{Q_i,W_i})=(\rmMod \hat{J}_{Q,W})^{(T_i[-1],F_i)}.$$
Note that we have an inclusion
\begin{equation}\label{eq:inclusion}T_i\cap\rmmod \hat{J}_{\tilde{Q},W}\subset \cC_{[1,\dots,n]},\end{equation}
which follows from the proof of Theorem \ref{Nagao}.

Take the central charge $Z_r$ on $Q,$ and an angle $\phi_r$ from Theorem \ref{torsion_pairs_via_Z}. Then, the DT series $A^{Z_r}_{V_{>\phi_r}}$ is well-defined, and it belongs to the completion of $\cT_{[1,\dots,n]}^{mot},$
hence we may and will treat it as an element of the completion of $\cT_{\Lambda}^{mot}.$

For convenience, we put
$$\hat{\Gamma}_{\un{k},\lambda}=\bigoplus_{j=1}^m\hat{\Gamma}_{\un{k},j}^{\oplus\lambda^j},\quad \hat{\Gamma}_{\un{k},j}=\Phi(r)^{-1}(\Gamma_{Q_r,W_r,j}).$$

\begin{theo}\label{cluster_var_via_conj}Let $M_i=\mu_{k_i}(\dots(\mu_{k_1}(M))\dots):\Z^m\to\cF_{\lambda}$ be the mutated toric frame. Then, for any $\lambda\in\Z_{\geq 0}^m,$
we have the equality
$$M_r^{mot}(\lambda)=A^{Z_r}_{V_{>\phi_r}}\cdot X^{[\hat{\Gamma}_{\un{k},\lambda}]}\cdot (A^{Z_r}_{V_{>\phi_r}})^{-1}.$$
\end{theo}

\begin{proof} First, the statement reduces to the case when $\lambda_j=\delta_{kj}$ for some $1\leq k\leq m.$ So, we will assume that this is the case.
By Theorem \ref{torsion_pairs_via_Z}, we have that $$A^{Z_r}_{V_{>\phi_r}}=(-T^{\frac12}\hat{w}_{[S(1)]};T)_{\infty}^{\epsilon_1}\dots (-T^{\frac12}\hat{w}_{[S(r)]};T)_{\infty}^{\epsilon_r}.$$
We will prove by induction on $0\leq i\leq r$ that
\begin{multline}\label{cluster_conj}M_i^{mot}(e_k)=(-T^{\frac12}\hat{w}_{[S(1)]};T)_{\infty}^{\epsilon_1}\dots (-T^{\frac12}\hat{w}_{[S(i)]};T)_{\infty}^{\epsilon_i}\cdot\\
\cdot X^{\Phi(i)^{-1}([\hat{\Gamma}_{\Q_i,W_i,k}])}\cdot
(-T^{\frac12}\hat{w}_{[S(i)]};T)_{\infty}^{-\epsilon_i}\dots (-T^{\frac12}\hat{w}_{[S(1)]};T)_{\infty}^{-\epsilon_1}.\end{multline}

For $i=0,$ there is nothing to prove. Suppose that \eqref{cluster_conj} is proved for some $i=l,$ $0\leq l<r.$
Then \eqref{cluster_conj} holds for $i=l+1,$ $k\ne k_{l+1},$ since $$M_l(e_k)=M_{l+1}(e_k),\quad [\hat{w}_{[S(l+1)]},X^{\Phi(l+1)^{-1}([\hat{\Gamma}_{\Q_{l+1},W_{l+1},k}])}]=[\hat{w}_{[S(l+1)]},X^{\Phi(l)^{-1}([\hat{\Gamma}_{Q_l,W_l,k}])}]=0.$$
just because
$$\Lambda(\Phi(l)^{-1}([\hat{\Gamma}_{\Q_l,W_l,k}]),[\iota(S(l+1))])=
\chi(\Phi(l)^{-1}([\hat{\Gamma}_{\Q_l,W_l,k}]),[S(l+1)])=\pm\delta_{k,k_{l+1}}=0.$$
Now, for the case $k=k_{l+1},$ we need the following Lemma.

\begin{lemma} Let $x,y$ be some formal variables, satisfying
$$xy=q^{\epsilon}yx,\quad\epsilon=\pm 1.$$
Then we have $$(-q^{\frac12}x;q)_{\infty}^{-\epsilon}y (-q^{\frac12}x;q)_{\infty}^{\epsilon}=y(1+q^{\frac{\epsilon}2}x).$$
\end{lemma}

\begin{proof}Indeed, it suffices to note that
$$(1+q^{\frac{d}2}x)y=y(1+q^{\frac{d}2-\epsilon}x).$$\end{proof}

Applying the above Lemma to $x=\hat{w}_{[S(l+1)]},$ $y=X^{\Phi(l+1)^{-1}([\hat{\Gamma}_{\Q_{l+1},W_{l+1},k_{l+1}}])},$
and $q^{\frac12}$ replaced by $T^{\frac12},$ we obtain the following equality:
\begin{multline*}(-\hat{w}_{[S(l+1)]};T)_{\infty}^{\epsilon_i}X^{\Phi(l+1)^{-1}([\hat{\Gamma}_{\Q_{l+1},W_{l+1},k_{l+1}}])}
(-\hat{w}_{[S(l+1)]};T)_{\infty}^{-\epsilon_i}=\\
X^{\Phi(l)^{-1}(-[\hat{\Gamma}_{Q_l,W_l,k_{l+1}}]+\sum\limits_{(\tilde{B}_l)_{j,k_{l+1}}>0}(\tilde{B}_l)_{j,k_{l+1}}[\hat{\Gamma}_{Q_l,W_l,j}])}+\\
X^{\Phi(l)^{-1}(-[\hat{\Gamma}_{Q_l,W_l,k_{l+1}}]-\sum\limits_{(\tilde{B}_l)_{j,k_{l+1}}<0}(\tilde{B}_l)_{j,k_{l+1}}
[\hat{\Gamma}_{Q_l,W_l,j}])}.\end{multline*}
Then, applying the inductive assumption, we conclude that \eqref{cluster_conj} holds for $i=l+1,$ $k=k_{l+1}.$ This proves Theorem.
\end{proof}

The central charge $Z_r$ on $Q$ defines the central charge $Z_r'$ on $Q_r,$ namely
$$Z_r'(\gamma)=\exp((\pi-\phi_r)\sqrt{-1})Z([\Phi(r)]^{-1}(\gamma)).$$

\begin{theo}\label{LP_by_sfr}For any $\lambda\in\Z_{\geq 0}^m$ we have \begin{equation}\label{LP_via_sfr}M_r^{mot}(\lambda)=X^{[\hat{\Gamma}_{\un{k},\lambda}]}\sum\limits_{\gamma\in\Z_{\geq 0}^{m}}[D(H^{\bullet,crit}_c(\cM_{\gamma,<(\pi-\phi_r),\lambda}^{sp,sfr},(W_r)_{\gamma}))]\cdot T^{-\frac12 \chi_{Q_r}(\gamma,\gamma)}\cdot
X^{\iota(\Phi(r)^{-1}[1](\gamma))}.\end{equation}
Moreover, the RHS is actually a finite sum.

In particular, we have that
\begin{equation}\label{Tate_type}[H^{\bullet,crit}_c(\cM_{\gamma,<(\pi-\phi_r),\lambda}^{sp,sfr},(W_r)_{\gamma})]\in \Z[T^{\pm 1}]\subset R\end{equation}
for all $\gamma.$\end{theo}

\begin{proof}First we show that the sum in the RHS is finite. Take the projective $\hat{J}_{Q_r,W_r}$-module
$\hat{J}_{Q_r,W_r,\lambda}=\bigoplus\limits_{j=1}^m \hat{J}_{Q_r,W_r,j}^{\oplus\lambda_j}.$ Then the pair
$((E\in M^{sp}_{\gamma,V_{<(\pi-\phi_r)}},u:\hat{J}_{Q_r,W_r,\lambda}\to E))$ defines a point in $M^{sp,sfr}_{\gamma,\pi-\phi_r,\lambda}$ iff the map
$$H^1(\Phi(r)^{-1}(u)):H^1(\Phi(r)^{-1}(\hat{J}_{Q_r,W_r,\lambda}))\to \Phi(r)^{-1}(E)[1]$$
is surjective. By Corollary \ref{fin_dim_of_H^1}, $H^1(\Phi(r)^{-1}(\hat{J}_{Q_r,W_r,\lambda}))=H^1(\hat{\Gamma}_{\un{k},\lambda})$ is finite-dimensional. In particular, $M^{sp,sfr}_{\gamma,\pi-\phi_r,\lambda}$ can be non-empty only if
$$0\leq\un{\dim}(\Phi(r)^{-1}(E)[1])\leq \un{\dim}\,H^1(\Phi(r)^{-1}(\hat{J}_{Q_r,W_r,\lambda})).$$

This implies that the sum in the RHS of \eqref{LP_via_sfr} is finite. By \eqref{eq:inclusion}, we have that
$$\Phi(r)^{-1}(E)[1]\in\cC_{[1,\dots,n]}.$$ Hence, for all $\gamma$ which contribute to the RHS of \eqref{LP_via_sfr}, we have that
$$\Lambda([\hat{\Gamma}_{\un{k},\lambda}],\Phi(r)^{-1}[1](\gamma))=
\sum\limits_{i=1}^m\lambda^i\gamma^i.$$

Thus, \eqref{LP_via_sfr} follows immediately from Theorem \ref{cluster_var_via_conj}, Theorem \ref{DT_for_formal}, 2), Remark \ref{parity_for_mutations}
and Proposition \ref{conj_and_sfr}.

Further, it follows that
$$[H^{\bullet,crit}_c(\cM_{\gamma,<(\pi-\phi_r),\lambda}^{sp,sfr},(W_r)_{\gamma})]\in\Z[T^{\pm \frac12}]\cap\hat{R}=\Z[T^{\pm 1}]\subset \hat{R}[T^{\frac12}],$$
so \eqref{Tate_type} holds.
\end{proof}

\section{Positivity conjecture via purity}
\label{s:PCVP}

\begin{defi}A graded object $H^{\bullet}$ in the category $MMHS$ (with $H^n=0$ for $|n|>>0$) is said to admit a Lefschetz operator centered at $N,$
if there is a morphism
$$L:H^{\bullet}\to H^{\bullet}(1)[2]$$
which induces isomorphisms
$$L^k:H^{N-k}\to H^{N+k}(k)$$
for all $k\in\Z_{>0}.$\end{defi}

\begin{cor}\label{purity->positivity}In the assumptions of Theorem \ref{LP_by_sfr}, suppose that for some $\lambda$ and for all $\gamma$ we have that
$H^{i,crit}_c(\cM_{\gamma,\pi-\phi_r,\lambda}^{sp,sfr},(W_r)_{\gamma})$
is pure of weight $i.$ Then the positivity conjecture holds for $M_r(\lambda)$.

Suppose that, moreover, that for some $\lambda$ and for all $\gamma$ the graded object $H^{\bullet,crit}_c(\cM_{\gamma,\pi-\phi_r,\lambda}^{sp,sfr},(W_r)_{\gamma})$
admits a Lefschetz operator. Then Conjecture \ref{conj_Lefschetz} holds for $M_r(\lambda).$\end{cor}

\begin{proof}Our assumption, together with \eqref{Tate_type}, immediately implies that $$H^{2i-1,crit}_c(\cM_{\gamma,\pi-\phi_r,\lambda}^{sp,sfr},(W_r)_{\gamma})=0,$$ $$H^{2i,crit}_c(\cM_{\gamma,\pi-\phi_r,\lambda}^{sp,sfr},(W_r)_{\gamma})=\Q(-i)^{\oplus a_{i,\gamma}},\,a_{i,\gamma}\geq 0,$$
hence $$M_r^{mot}(\lambda)=X^{[\hat{\Gamma}_{\un{k},\lambda}]}\cdot\sum\limits_{i,\gamma}a_{i,\gamma}T^{-i-\frac12\chi_{Q_r}(\gamma,\gamma)} X^{\iota(\Phi(r)^{-1}[1](\gamma))},\quad a_{i,\gamma}\in\Z_{\geq 0.}$$

The second assertion is then clear.\end{proof}

For any representation $E$ of a quiver $Q,$ and a dimension vector $\gamma\in\Z^{V(Q)},$ the quiver Grassmannian
$\Gr(E,\gamma)$ is the scheme of subrepresentations $E'\subset E$ with $\un{\dim}(E/E')=\gamma.$ Clearly,
$\Gr(E,\gamma)$ is a projective scheme.

\begin{prop}\label{projectivity} In the assumptions of Theorem \ref{LP_by_sfr},
Assume that the subvariety $$\cM_{\gamma,<(\pi-\phi_r),\lambda}^{sp,sfr}\subset \cM_{\gamma,<(\pi-\phi_r),\lambda}^{sfr}$$
is a union of connected components of the subvariety $Crit((W_r)_{\gamma}).$ Then we can treat $\cM_{\gamma,<(\pi-\phi_r),\lambda}^{sp,sfr}$
 as a union of connected components of the scheme $Crit((W_r)_{\gamma}).$ Then there is an isomorphism of schemes
$$\cM_{\gamma,<(\pi-\phi_r),\lambda}^{sp,sfr}\cong \Gr(H^1(\hat{\Gamma}_{\un{k},\lambda}),\Phi(r)^{-1}[1](\gamma)).$$
In particular, the scheme $\cM_{\gamma,<(\pi-\phi_r),\lambda}^{sp,sfr}$ is projective.\end{prop}

\begin{proof}We have already seen in the proof of Theorem \ref{LP_by_sfr} that there is a bijection
between closed points of these schemes. It is easy to show that it is induced by an isomorphism of schemes.\end{proof}

\begin{remark} In \cite{P1}, Plamondon obtains a general formula for cluster monomials in commutative cluster algebras.
The same formulas are actually obtained in \cite{DWZ2}, the coincidence is shown in \cite{P2}.
The resulting coefficients are Euler characteristics of some quiver Grassmannians. By Corollary \ref{fin_dim_of_H^1},
these quiver Grassmannians are precisely $\Gr(H^1(\hat{\Gamma}_{\un{k},\lambda}),\Phi(r)^{-1}[1](\gamma)).$
On the other hand, the Euler characteristics of critical cohomology (under the assumption of Proposition \ref{projectivity}) coincides with Behrend's weighted Euler characteristics \cite{Be}
$$\tilde{\chi}(\cM_{\gamma,<(\pi-\phi_r),\lambda}^{sp,sfr})=\tilde{\chi}(\Gr(H^1(\hat{\Gamma}_{\un{k},\lambda}),\Phi(r)^{-1}[1](\gamma))),$$
where for any scheme $X$ the weighted Euler characteristics is defined using constructible function (Behrend's function)
$$\nu_X:X\to\Z.$$ This motivates the following question.
\end{remark}

{\noindent{\bf Question. }}{\it Let $(Q,W)$ be a formal QP, and $M\in\rmmod\hat{J}_{Q,W}$ is such that the decorated representation $(M,0)$
has $E$-invariant zero (see \cite{DWZ},\cite{P2}). Is it true that for all $\gamma\in\Z_{\geq 0}^{V(Q)}$ the Behrend's function
$$\nu:\Gr(E,\gamma)\to\Z$$
is identically equal to $1$?}

\begin{prop}\label{purity_of_MHM}In the assumptions of Theorem \ref{LP_by_sfr}, suppose that for all $\gamma$ and for generic
$W_r$ (possibly depending on $\gamma$) the subvariety $$\cM_{\gamma,<(\pi-\phi_r),\lambda}^{sp,sfr}\subset \cM_{\gamma,<(\pi-\phi_r),\lambda}^{sfr}$$
is a union of connected components of $Crit((W_r)_{\gamma}),$ and assume that the restriction of
$$\phi_{\frac{(W_r)_{\gamma}}u}\Q_{(\cM_{\gamma,\pi-\phi_r,\lambda}^{sfr})\times G_m}(0)[\sum\limits_{j=1}^m\lambda^i\gamma^i-\chi_{Q_r}(\gamma,\gamma)+1]$$
onto $\cM_{\gamma,\pi-\phi_r,\lambda}^{sp,sfr}$ is a pure Hodge module of weight $\sum\limits_{j=1}^m\lambda^i\gamma^i-\chi_{Q_r}(\gamma,\gamma)+1.$
Then the critical cohomology $H^{i,crit}_c(\cM_{\gamma,\pi-\phi_r,\lambda}^{sp,sfr},(W_r)_{\gamma})$ is pure of weight $i,$
and the graded object $H^{\bullet,crit}_c(\cM_{\gamma,\pi-\phi_r,\lambda}^{sp,sfr},(W_r)_{\gamma})$ admits a Lefschetz operator
centered at $\sum\limits_{j=1}^m\lambda^i\gamma^i-\chi_{Q_r}(\gamma,\gamma).$

In particular, in this case the Conjecture \ref{conj_Lefschetz} holds for $M_r(\lambda).$
\end{prop}

\begin{proof}Indeed, by the general theory of M. Saito, the direct image for a proper morphisms preserves weights (for objects of the derived category), and the cohomology of the direct image by a projective morphism of a pure Hodge module admits a Lefschetz operator, see \cite{PS}.\end{proof}

\begin{theo}\label{positivity_for_acyclic}In the assumptions of Theorem \ref{LP_by_sfr}, suppose that either $Q_{[1,\dots,n]}$ or $(Q_r)_{[1,\dots,n]}$ is acyclic quiver.
Then for any $\lambda\in\Z_{\geq 0}^m,$ $\gamma\in\Z_{\geq 0}^n,$ we have that
the assumptions of Proposition \ref{purity_of_MHM} hold. In particular, Conjecture \ref{conj_Lefschetz} holds for all the cluster monomials
$M_r(\lambda).$\end{theo}

\begin{proof}First, consider the case when the quiver $(Q_r)_{[1,\dots,n]}$ is acyclic. Then for $\gamma\in\Z_{\geq 0}^n\subset\Z_{\geq 0}^m$ we have $(W_r)_{\gamma}=0,$ hence
$$\cM_{\gamma,\pi-\phi_r,\lambda}^{sp,sfr}=\cM_{\gamma,\pi-\phi_r,\lambda}^{sfr},$$
and we have that
\begin{multline*}\phi_{\frac{(W_r)_{\gamma}}u}\Q_{(\cM_{\gamma,\pi-\phi_r,\lambda}^{sfr})\times G_m}(0)[\sum\limits_{j=1}^m\lambda^i\gamma^i-\chi_{Q_r}(\gamma,\gamma)+1]=\\\Q_{(\cM_{\gamma,\pi-\phi_r,\lambda}^{sfr})\times G_m}(0)[\sum\limits_{j=1}^m\lambda^i\gamma^i-\chi_{Q_r}(\gamma,\gamma)+1],\end{multline*}
is certainly a pure Hodge module. The resulting critical cohomology is just the cohomology of a smooth projective variety, which is Hodge-Tate by \eqref{Tate_type}.

Now, consider the case when $Q_{[1,\dots,n]}$ is acyclic.
First, we claim that the variety $M_{\gamma}^{sp}$ is always a union of connected components of $Crit((W_r)_{\gamma})\cap (W_r)_{\gamma}^{-1}(0),$
if $\gamma\in\Z_{\geq 0}^n\subset\Z_{\geq 0}^m.$
Indeed, this holds for generic polynomial potential $W_r$ by Proposition \ref{Bertini}. Further, since there are no non-trivial potentials on $Q_{[1,\dots,n]},$ all possible potentials $W_r$ are equivalent when restricted to $(Q_r)_{[1,\dots,n]}.$ Hence, the corresponding
functions $(W_r)_{\gamma}$ are obtained from each other by an automorphism of the formal neighborhood of $Nilp_{\gamma}\subset M_{\gamma}.$
It follows that for any possible potential $M_{\gamma}^{sp}$ is a union of connected components of $Crit((W_r)_{\gamma})\cap W_{\gamma}^{-1}(0).$

 We have that
locally at each point $p\in \cM_{\gamma,\pi-\phi_r,\lambda}^{sp,sfr}$ the function $W_{\gamma}$
is analytically equivalent to a quadratic form of even rank
$$\chi_{Q}(\Phi(r)^{-1}[1](\gamma),\Phi(r)^{-1}[1](\gamma))-\chi_{Q_r}(\gamma,\gamma).$$
Indeed, if the point $p$ is given by $(E,u),$ then the minimal potential on $$\Ext^1(\Phi(r)^{-1}[1](E),\Phi(r)^{-1}[1](E))\cong\Ext^1(E,E)$$
equals to zero by Theorem \ref{Kajiura}, and from the same Theorem it follows that $W_{\gamma}$ in the neighborhood of $p$ is formally, hence analytically,
equivalent to a quadratic form of rank as in the above formula. The rank is even by Remark \ref{parity_for_mutations}.

Therefore, the restriction of the mixed Hodge module $\phi_{\frac{(W_r)_{\gamma}}u}\Q_{(\cM_{\gamma,\pi-\phi_r,\lambda}^{sfr})\times G_m}(0)[\sum\limits_{j=1}^m\lambda^i\gamma^i-\chi_{Q_r}(\gamma,\gamma)+1]$ onto $\cM_{\gamma,\pi-\phi_r,\lambda}^{sp,sfr}\times G_m$
is a pure Hodge module of weight $\sum\limits_{j=1}^m\lambda^i\gamma^i-\chi_{Q_r}(\gamma,\gamma).$\end{proof}

\begin{remark}In the case when $Q_{[1,\dots,n]}$ is acyclic, the positivity conjecture (and also Conjecture \ref{conj_Lefschetz}) has already been shown by F. Qin \cite{Q}. He interprets quantum cluster monomials via Serre polynomials of quiver Grassmannians, and these quiver Grassmannians
are actually $\Gr(H^1(\hat{\Gamma}_{\un{k},\lambda}),\Phi(r)^{-1}[1](\gamma)).$ The cohomology of these Hodge-Tate smooth projective varieties
actually coincide (up to a twist and a shift) with our critical cohomology.\end{remark}

The following conjecture was suggested to me by M. Kontsevich.

\begin{conj}\label{purity_conj}In the above notation, for the generic potential $W_r$ (possibly depending on $\gamma$) the
critical cohomology
$H^{i,crit}_c(\cM_{\gamma,\pi-\phi_r,\lambda}^{sp,sfr},(W_r)_{\gamma})$
is pure of weight $i,$ and admits a Lefschetz operator.\end{conj}

We remark that even if the set $\cM_{\gamma,\pi-\phi_r,\lambda}^{sp,sfr}$ coincides with $Crit((W_r)_{\gamma})\cap(W_r)_{\gamma}^{-1}(0),$ the mixed Hodge module $\phi_{\frac{(W_r)_{\gamma}}u}\Q_{(\cM_{\gamma,\pi-\phi_r,\lambda}^{sfr})\times G_m}(0)[\sum\limits_{j=1}^m\lambda^i\gamma^i-\chi_{Q_r}(\gamma,\gamma)+1]$ is not always pure.

Namely, consider the case $n=3,$ and take the initial quiver $Q$ such that the coefficient-free part $Q_{[1,2,3]}$
is just a cycle of length $3,$ i.e.
$$a_{12}=a_{23}=a_{31}=1, a_{11}=a_{22}=a_{33}=a_{21}=a_{32}=a_{13}=0,$$
with the potential $W,$ when restricted to $Q_{[1,2,3]}$  equal to this cyclic path.
Then, taking the sequence of mutations $\un{k}=(1,2,3,1),$ we obtain the quiver $Q_4,$ for which
$(Q_4)_{[1,2,3]}$ is again a cycle of length $3,$ but in the opposite direction, and the potential $W_4$ (again, restricted to $(Q_4)_{[1,2,3]}$) is again this cyclic path. It is easy to check that the resulting torsion pair in $\rmmod\hat{J}_{Q,W}$ is precisely
$$(\cC_{[1,2,3]},\cC_{[1,2,3]}^{\perp}).$$ This means that for $\gamma\in\Z_{\geq 0}^3\subset\Z_{\geq 0}^m,$ $\lambda\in\Z_{\geq 0}^m$
we have
$$\cM_{\gamma,\pi-\phi_r,\lambda}^{sfr}=\{(E\in M_{\gamma},u:P_{\lambda}\to E)\mid u\text{ is surjective}\}/G_{\gamma}.$$

Let us describe this variety in the case
$$\lambda=\gamma=(1,1,1,0,\dots,0).$$ In this case it is a toric variety
$$X=U/(\C^*)^3,$$
where $U\subset \A^6_{u_1,u_2,u_3,x_{12},x_{23},x_{31}}$ is an open subset
given by
$$(u_1,x_{31})\ne (0,0),\quad (u_2,x_{12})\ne (0,0),\quad (u_3,x_{23})\ne (0,0),\quad (u_1,u_2,u_3)\ne (0,0,0).$$
The action of $(\C^*)^3$ is described by the formula
$$(t_1,t_2,t_3)\cdot (u_1,u_2,u_3,x_{12},x_{23},x_{31})=(t_1u_1,t_2u_2,t_3u_3,t_2x_{12}t_1^{-1},t_3x_{23}t_2^{-1},t_1x_{31}t_3^{-1}).$$
The potential $f:X\to\C$ is given by the formula
$$f=x_{12}x_{23}x_{31}.$$ We have
$$X^{sp}=Crit(f)\cap f^{-1}(0)=Crit(f)=\{x_{23}=x_{31}=0\}\cup\{x_{31}=x_{12}=0\}\cup\{x_{12}=x_{23}=0\}$$
--- a union of three projective lines intersecting at one point $$p=((\C^*)^3\times(0,0,0))/(\C^*)^3.$$

Since the weight filtration is the subject of Betti realization, we can consider the perverse sheaves of $\Q$-vector spaces, and
take the vanishing cycles functor $\phi_f$ without multiplying by $G_m.$ Then, the perverse sheaf $\phi_f\Q_X[3]\in\Perv(X^{sp})$ has
three non-zero subquotients in the weight filtration:
$$\gr^{W}_2\phi_f\Q_X[3]\cong \gr^{W}_4\phi_f\Q_X[3]\cong\Q_p,\quad \gr^{2}_3\phi_f\Q_X[3]\cong IC(X^{sp}),$$
$$\gr^{W}_n\phi_f\Q_X[3]=0\text{ for }n\ne 2,3,4.$$
where $Q_p$ is the skyscraper sheaf at the point $p,$ and $IC(X^{sp})$ is the intersection cohomology (perverse) sheaf.
In particular, the sheaf $\phi_f\Q_X[3]$ is not pure. However, the critical cohomology is pure:
$$H^{2,crit}_c(X^{sp},W)\cong\Q(-1)^{\oplus 2},\quad H^{4,crit}_c(X^{sp},W)\cong\Q(-2)^{\oplus 2},$$
$$H^{n,crit}_c(X^{sp},W)=0\text{ for }n\ne 2,4.$$

\begin{remark}Our example of non-purity of the MHM of vanishing cycles is similar to (and much simpler than) \cite{DS}.
In both cases the weight filtration has three subquotients, two beeing skyscraper sheafs and the third one is the intersection cohomology sheaf on the
subscheme of singularities of the zero fiber.\end{remark}

\section{A conjecture on exceptional collections}
\label{s:conj_on_exc}

In this section we propose a conjecture on existence of exceptional collection in certain $2$-periodic triangulated categories
related to the moduli of stable framed representations.

\begin{defi} A triangulated $k$-linear category $\cC$ is called $2$-periodic if there is a fixed isomorphism of functors
$$\id\cong [2].$$

Let $\cC$ be a $2$-periodic category with finite-dimensional $\Hom$-spaces. An object $E\in\cC$ is called exceptional if
$$\Hom^1(E,E)=0,\quad \Hom^0(E,E)=k.$$

A sequence $(E_1,\dots,E_n)$ of exceptional objects in a $2$-periodic triangulated category $\cC$ is called an exceptional collection if
$$\Hom^{\bullet}(E_i,E_j)=0,\quad i>j.$$ An exceptional collection is called full if it generates the triangulated category $\cC$
\end{defi}

Let $(X,f)$ be a smooth algebraic variety with a regular function. Then one can associate with $(X,f)$ two equivalent $2$-periodic triangulated categories. The first one is the category of singularities \cite{Or1} of the zero fiber $X^0=f^{-1}(0):$
$$D_{sg}(X^0):=D^b_{coh}(X^0)/\Perf(X^0).$$
The second one is the homotopy category of the D($\Z/2$-)G category of matrix factorizations $MF(X,f)$ \cite{Or2}, \cite{PV}, \cite{LP},
\cite{Pos}.

By the theorem of Orlov \cite{Or2}, we have
an equivalence
$$D_{sg}(X^0)\cong\Ho(MF(X,f)).$$
This in particular explains that the category $D_{sg}(X^0)$ is $2$-periodic. In general these (equivalent) categories are not Karoubi complete,
and we consider the Karoubian completion $D_{sg}(X^0)^{\kappa}.$

\begin{conj}Let $X$ be some moduli space of stable framed representations $\cM_{\gamma,\pi-\phi_r,\lambda}^{sfr}$
which arise in Theorem \ref{LP_by_sfr}, and let $f=(W_r)_{\gamma}:X\to\C,$ where $W_r$ is generic polynomial potential. Then
the category $D_{sg}(X^0)^{\kappa}$ admits a full strong exceptional collection.
\end{conj}

This conjecture implies the purity conjecture, i.e. the first half of Conjecture \ref{purity_conj}, and hence the positivity conjecture.

\section{Appendix}
\label{s:Appendix}

In this Appendix we prove Theorem \ref{DT_for_formal}.

Let us consider the following situation. Let $A$ be a $3$CY algebra, i.e. an $A_{\infty}$-algebra with scalar product
of degree $3.$ Let us assume that
$$m_n=0\text{ for }n> N.$$
Then the potential defined above is the well-defined polynomial function,
$$W_A:A^1\to\C.$$
The set of its critical points is precisely the set of Maurer-Cartan solutions:
$$Crit(W_A)=\{dW_A=0\}=\{\sum\limits_{n=1}^N m_n(\alpha,\dots,\alpha)=0\}=MC(A)\subset A.$$
The points of this set are objects of the $3$CY $A_{\infty}$-category, which we denote by $\cM\cC(A)_{\infty}$ (and which is defined exactly as
 the $3$CY $A_{\infty}$-category of twisted complexes).

Suppose that we are given with some full $A_{\infty}$ subcategory $\cC\subset \cM\cC(A)_{\infty},$ such that

$(\star)$ the subset $Ob(\cC)\subset Ob(\cM\cC(A)_{\infty})\cap W_A^{-1}(0)$ is constructible.

We would like to define the class
$$[D(H^{\bullet,crit}_c(\cC,W_A))].$$ This is done as follows. Choose some constructible subset
$$\bigsqcup\limits_{i=1}^m Y_i\subset Ob(\cC),$$ with locally closed $Y_i,$ such that:

$(i)$ Each isomorphism class of objects in $\Ho(\cC)$ contains exactly one representative in $\bigsqcup\limits_{i=1}^m Y_i;$

$(ii)$ For a fixed $1\leq i\leq m,$ and for $X\in Y_i,$ we have
$$\Hom^0_{\Ho(\cC)}(X,X)/J(X)\cong\bigoplus_{j=1}^t \Mat_{n_j}(\C),$$
where $J(X)$ is the Jacobson radical of the algebra $\Hom^0_{\Ho(\cC)}(X,X),$ and the numbers $t,$ $n_1,\dots,n_t,$ and $\dim J(X)$ do not depend on $X\in Y_i.$
In particular, the class $[H^{\bullet}_c(\Aut_{\Ho(\cC)}(X))]\in K_0(MHS)$ does not depend on $X\in Y_i.$

This allows us to define the class
\begin{equation}\label{H^crit}[D(H^{\bullet,crit}_c(\cC,W_A))]=\sum\limits_{i=1}^m\frac{[D(H^{\bullet,crit}(Y_i,W_A))]}{[D(H^{\bullet}_c(\Aut_{\Ho(\cC)}(X_i)))]}
\in \hat{R},\end{equation}
where $X_i\in Y_i$ is some object.

For convenience, for the finite-dimensional graded vector space $V,$ put
$$\chi_{\leq d}(V)=\sum\limits_{i\leq d}(-1)^i\dim V^i.$$

\begin{prop}\label{well-defined_class}The class \eqref{H^crit} does not depend on the choice of the constructible subset
$\bigsqcup\limits_{i=1}^m Y_i\subset Ob(\cC).$\end{prop}

\begin{proof} This class is actually dual to the Hodge realization of the motivic Milnor fiber used in \cite{KS08}.
Namely, for any smooth complex algebraic variety $X$ with a regular function $W$ Denef and Loeser (\cite{DL1},\cite{DL2}) define
the element $S_W\in M^{\mu}(X^0),$ where $X^0=W^{-1}(0),$ $\mu=\lim\limits_{\leftarrow}\mu_n,$ and for a scheme $Y$ one defines $M^{\mu}(Y)$
to be the localized Grothendieck group of $Y$-schemes with good action of $\mu.$ Further, for any locally closed embedding $\iota:Z\hookrightarrow X_0,$
denoting by $p:Z\to pt$ the projection we get the element
$$S_{W,Z}:=p_{!}\iota^* S_W\in M^{\mu}_{\C}:=M^{\mu}(\Spec\C).$$

There is a natural map $$F:M^{\mu}_{\C}\to K_0(MMHS).$$
If $X$ is an algebraic variety with good action of some $\mu_n,$ then one puts
$$F([X])=(G_m\to\A^1)_{!}((X\times G_m)/\mu_n\stackrel{g}{\to}G_m)_!\Q_{(X\times G_m)/mu_n}(0).$$
The map $F$ is actually a homomorphism, where one considers the Thom-Sebastiani product on $M^{\mu}_{\C}$ \cite{DL3}.
According to \cite{KS08}, we have the equality
$$F(1-S_{W,Z})=[H^{\bullet,crit}_c(Z,W)].$$

Kontsevich and Soibelman \cite{KS08} use more general motivic Milnor fibre for formal power series. Put $\sqrt{\bL}:=1-S_{z^2,0}.$
Let$\bigsqcup\limits_{i=1}^m Y_i\subset Ob(\cC)$ be some constructible subset as above, and fix some $Y_i.$
Then for $\alpha\in Y_i$ the potential on $\End^1(\alpha)\cong A^1$ equals to $W_{\alpha}(z)=W_A(z+\alpha).$ It follows from Theorem \ref{Kajiura}
and the Thom-Sebastiani Theorem \cite{DL3} that we have
$$S_{W_A,Y_i}=\int\limits_{\alpha\in Y_i}(S_{W_{\alpha}^{min},0}\cdot\sqrt{\bL}^{\dim \End^1(\alpha)/\ker(m_1:\End^1(\alpha)\to\End^2(\alpha))}).$$
Since the potential $W_{\alpha}^{min}$ depends only on the isomorphism class of $\alpha$ in $\Ho(\cC),$
and we have $$\dim \End^1(\alpha)/\ker(m_1:\End^1(\alpha)\to\End^2(\alpha))=\chi_{\leq 1}(\Hom^{\bullet}_{\Ho(\cC)}(\alpha,\alpha))-\chi_{\leq 1}(A),$$
we conclude that the element
$S_{W,Y_i}$ depends only on the set of homotopy equivalence classes of objects in $Y_i.$ Hence,
the class $[D(H^{\bullet,crit}_c(\cC,W_A))]$ is well-defined.
\end{proof}

Now, suppose that we have some $3$CY $A_{\infty}$-category $\cA$ with two objects $U_1,U_2,$ with $\End(U_i)=A_i,$
and again only finitely many of $m_n$ are non-zero on $\cA.$
Then, we have full $A_{\infty}$-subcategories $\cM\cC(A_1)_{\infty},\cM\cC(A_2)_{\infty}\subset\cM\cC(\cA)_{\infty}.$ Assume that
$\cC_i\subset\cM\cC(A_i)_{\infty}$ are $A_{\infty}$-subcategories satisfying $(\star).$ Moreover, let us assume that
$$\Hom^{\leq 0}_{\Ho(\cM\cC(\cA)_{\infty})}(X_1,X_2)=0,\quad \Hom^{<0}_{\Ho(\cM\cC(\cA)_{\infty})}(X_2,X_1)=0,\quad X_1\in Ob(\cC_1),\,X_2\in Ob(\cC_2).$$

Then define the subcategory $\cC_{1,2}\subset\cM\cC(\cA)_{\infty},$
$$\cC_{1,2}=\{\alpha=(\alpha_{11},\alpha_{12},\alpha_{21},\alpha_{22})\in MC(\cA)\mid \alpha_{12}=0,\alpha_{11}\in Ob(\cC_1),\alpha_{22}\in Ob(\cC_2)\}.$$

\begin{prop}\label{factorization}We have that
$$[D(H^{\bullet,crit}_c(\cC_{1,2}),W_{\cA})]=[D(H^{\bullet,crit}_c(\cC_{1}),W_{A_1})]\cdot [D(H^{\bullet,crit}_c(\cC_{2}),W_{A_2})]\cdot T^{\sum\limits_{i\leq 1}\dim\Hom_{\cC}(U_2,U_1)}.$$\end{prop}

\begin{proof}This follows from the Integral Identity proved by in \cite{KS}, Subsection 7.8. The implication is similar to \cite{KS}, Subsection 7.7.\end{proof}

We formulate once again Theorem \ref{DT_for_formal}.

\begin{theo}Suppose that for some formal QP $(Q,W)$ and polynomial QP $(Q',W')$ we have a cyclic $A_{\infty}$-functor
$\phi:\cC_{Q,W}\to Tw\,\cC_{Q',W'},$ inducing an equivalence
$\Phi:D^b(\hat{\Gamma}_{Q,W})\stackrel{\sim}{\to} D^b(\hat{\Gamma}_{Q',W'}).$ Let $Z$ be a central charge on $Q.$ Then one can define the classes
$$[D(H^{\bullet,crit}_{c,G_{\gamma}}(M_{\gamma,V})^{sp},W_{\gamma})]\in\hat{R},$$
for all sectors $V\subset \cH_+,$ and $\gamma\in\Z_{\geq 0}^{V(Q)},$ such that the following holds.

1) Suppose that there is only one $G_{\gamma}$-orbit in $M_{\gamma,V}^{sp},$ and for the corresponding representation $E$
of $\hat{J}_{Q,W}$ we have $$\Ext^1(E,E)=0,\quad \chi_Q(\gamma,\gamma)\equiv \dim\Ext^0(E,E)\text{ mod }2.$$
Then we have
\begin{equation}\label{crit_cohom_formal2}[D(H^{\bullet,crit}_{c,G_{\gamma}}(M_{\gamma,V})^{sp},W_{\gamma})]=[H^{\bullet}(\mrB\Aut(E))]\cdot T^{\dim_{\C}\Aut(E)}.\end{equation}

2) Define the DT series $A_V$ using the classes $[D(H^{\bullet,crit}_{c,G_{\gamma}}(M_{\gamma,V})^{sp},W_{\gamma})]\in\hat{R},$
as in the formula \eqref{formula_for_A_V}.

Suppose that we have a central charge $Z'$ on $Q',$ and for some sectors $V,V'\subset\cH_+$ we have that
$$\Phi(\cC_V)=\cC_{V'},\quad \cC_V\subset D^b(\hat{\Gamma}_{Q,W}),\,\cC_{V'}\subset D^b(\hat{\Gamma}_{Q',W'}),$$ and assume that
\begin{equation}\label{parity_preserved2}\chi_{Q'}([\Phi](\gamma),[\Phi](\gamma))\equiv \chi_Q(\gamma,\gamma)\text{ mod }2,\quad\gamma\in\Z^{V(Q)}.\end{equation}
Then we have
$$A_{V'}=[\Phi](A_V),$$
where in the last formula $[\Phi]$ denotes the induced map on completions of motivic quantum tori.

3) If the sector $V$ is the disjoint union of two sectors $V_1\sqcup V_2$ (in the clockwise order), then we have factorization:
$$A_V=A_{V_1}A_{V_2}.$$
\end{theo}

\begin{proof}
We have the objects $\phi(S_i)\in Tw\,\cC_{Q',W'},$ $i\in V(Q).$ For any $\gamma\in\Z_{\geq 0}^{V(Q)},$
put $$X_{\gamma}:=\bigoplus\limits_{i\in V(Q)}\phi(S_i)^{\oplus\gamma^i}\in Tw\, \cC_{Q',W'}.$$
Put $$\cA_{\gamma}:=\End_{Tw\,\cC_{Q',W'}}(X_{\gamma}).$$
For any central charge $Z$
on $Q,$ and a sector $V\subset\cH_+,$ take the subcategory $\cC_{\gamma,V}\subset\cM\cC(\cA_{\gamma})_{\infty},$ corresponding to MC solutions $\alpha\in\cA_{\gamma}^1,$
which are strictly upper-triangular with respect to some order of direct summands of $X_{\gamma},$ and such that the corresponding object of $D^b(\hat{\Gamma}_{Q,W})$
(under the preimage $\Phi^{-1}$) is a representation of $\hat{J}_{Q,W}$ from $M_{\gamma,V}^{sp}.$
Define $$[D(H^{\bullet,crit}_{c,G_{\gamma}}(M_{\gamma,V}^{sp},W))]:=[D(H^{\bullet,crit}_c(\cC_{\gamma,V}),W_{\cA_{\gamma}})]\cdot\sqrt{T}^{\chi_{Q}(\gamma,\gamma)- \chi_{\leq 1}(\cA_{\gamma})}.$$

The property 1) follows directly from the definition.

To show 2), first put
$$A_V=1+\sum\limits_{\gamma\in C_V}[D(H_{c,G_{\gamma}}^{\bullet,crit}(M_{\gamma,V}^{sp},W_{\gamma}))]\cdot
T^{-\frac{\chi(\gamma,\gamma)}2}\hat{w}_{\gamma}\in\hat{\cT}_{Q,C_V}^{mot}.$$

Note that the DT series defined in \cite{KS} for polynomial potential (see formula \ref{formula_for_A_V} above) are actually obtained in the same way,
replacing $\phi$ by the inclusion
$\cC_{Q',W'}\to Tw\,\cC_{Q',W'}.$
Indeed, denote the objects of $\cC_{Q',W'}$ by $S_i',$ and put
$$X_{\gamma'}':=\bigoplus\limits_{i\in V(Q)}\phi(S_i')^{\oplus\gamma'^i}\in Tw\, \cC_{Q',W'},$$
$$\cA_{\gamma'}':=\End_{Tw\,\cC_{Q',W'}}(X_{\gamma'}'),$$
and define$\cC_{\gamma',V'}\subset \cM\cC(\cA_{\gamma'}')_{\infty}$ in the same way as above.
Then we have
\begin{multline*}[D(H^{\bullet,crit}_{c,G_{\gamma'}}(M_{\gamma',V'}^{sp},W'))]=[D(H^{\bullet,crit}_c(\cC_{\gamma',V'}),W_{\cA_{\gamma'}'})]=\\
[D(H^{\bullet,crit}_c(\cC_{\gamma',V'}),W_{\cA_{\gamma'}'})]\cdot\sqrt{T}^{\chi_{Q'}
(\gamma',\gamma')- \chi_{\leq 1}(\cA_{\gamma'}')}.\end{multline*}

Now assume that $\gamma'=[\Phi](\gamma).$ We need to show that
\begin{multline}\label{equality_of_crit_cohom}[D(H^{\bullet,crit}_c(\cC_{\gamma',V'}),W_{\cA_{\gamma'}'})]\cdot\sqrt{T}^{\chi_{Q'}
(\gamma',\gamma')- \chi_{\leq 1}(\cA_{\gamma'}')}\cdot T^{-\frac{\chi_{Q'}(\gamma',\gamma')}2}=\\
[D(H^{\bullet,crit}_c(\cC_{\gamma,V}),W_{\cA_{\gamma}})]\cdot\sqrt{T}^{\chi_{Q}(\gamma,\gamma)- \chi_{\leq 1}(\cA_{\gamma})}\cdot T^{-\frac{\chi_{Q}(\gamma,\gamma)}2}.\end{multline}

By \eqref{parity_preserved2}, we may replace $T^{\frac12}$ by $\sqrt{T}$ in \eqref{equality_of_crit_cohom}.
Then, arguing as in Proposition \ref{well-defined_class}, we express LHS and RHS of \eqref{equality_of_crit_cohom} using motivic Milnor fibres, and obtain the desired equality.

The property 3)
follows from Proposition \ref{factorization}, similarly to \cite{KS}.
\end{proof}

\end{document}